\documentclass{amsart}

\usepackage{amsmath,amsfonts,mathrsfs,amsthm,amssymb,stmaryrd}
\usepackage{array}
\usepackage{pdflscape}
\usepackage{tikz,tikz-cd}

\newtheorem{theorem}{Theorem}
\newtheorem{corollary}[theorem]{Corollary}
\newtheorem{proposition}[theorem]{Proposition}
\newtheorem{lemma}[theorem]{Lemma}
\newtheorem*{theoremA}  {Theorem A}

\theoremstyle{definition}
\newtheorem{definition}[theorem]{Definition}
\newtheorem{example}[theorem]{Example}
\newtheorem{remark}[theorem]{Remark}

\newcommand{\ab}[1]{ \left\langle {#1} \right\rangle }
\newcommand{\set}[1]{ \{ {#1} \} }

\newcommand{\imult}{\lrcorner \,}
\newcommand{\ot}   {\otimes}

\newcommand{\bbA}{\mathbb{A}}
\newcommand{\bbC}{\mathbb{C}}
\newcommand{\bbF}{\mathbb{F}}
\newcommand{\bbH}{\mathbb{H}}
\newcommand{\bbI}{\mathbb{I}}
\newcommand{\bbJ}{\mathbb{J}}
\newcommand{\bbO}{\mathbb{O}}
\newcommand{\bbP}{\mathbb{P}}
\newcommand{\bbR}{\mathbb{R}}
\newcommand{\bbS}{\mathbb{S}}
\newcommand{\bbT}{\mathbb{T}}
\newcommand{\bbV}{\mathbb{V}}
\newcommand{\bbZ}{\mathbb{Z}}

\newcommand{\mba}{\mathbf{a}}
\newcommand{\mbg}{\mathbf{g}}
\newcommand{\mbaEs}{\mathbf{aEs}}

\newcommand{\mcD}{\mathcal{D}}

\newcommand{\mcG}{\mathcal{G}}
\newcommand{\mcL}{\mathcal{L}}
\newcommand{\mcN}{\mathcal{N}}
\newcommand{\mcQ}{\mathcal{Q}}
\newcommand{\mcS}{\mathcal{S}}
\newcommand{\mcT}{\mathcal{T}}

\newcommand{\mfaut} {\mathfrak{aut}}
\newcommand{\mfg}   {\mathfrak{g}}
\newcommand{\mfh}   {\mathfrak{h}}
\newcommand{\mfk}   {\mathfrak{k}}
\newcommand{\mfo}   {\mathfrak{o}}

\newcommand{\mfgl}  {\mathfrak{gl}}
\newcommand{\mfsl}  {\mathfrak{sl}}

\newcommand{\mfhol} {\mathfrak{hol}}
\newcommand{\mfstab}{\mathfrak{stab}}

\newcommand{\olxi}{\overline{\xi}}

\newcommand{\whD}{\smash{\widehat{D}}}
\newcommand{\whg}{\smash{\widehat{g}}}
\newcommand{\whM}{\smash{\widehat{M}}}
\newcommand{\whV}{\smash{\widehat{V}}}
\newcommand{\whx}{\smash{\widehat{x}}}

\newcommand{\wtg}    {\smash{\widetilde{g}}}
\newcommand{\wtM}    {\smash{\widetilde{M}}}
\newcommand{\wtR}    {\smash{\widetilde{R}}}
\newcommand{\wtPhi}  {\smash{\widetilde{\Phi}}}
\newcommand{\wtnabla}{\smash{\widetilde{\nabla}}}

\newcommand{\whmfk}{\widehat{\mathfrak{k}}}
\newcommand{\wtbbO}{\smash{\widetilde{\mathbb{O}}}}

\DeclareMathOperator{\Ann} {Ann}
\DeclareMathOperator{\Aut} {Aut}

\DeclareMathOperator{\CO}  {CO}
\DeclareMathOperator{\dom} {dom}
\DeclareMathOperator{\End} {End}
\DeclareMathOperator{\G}   {G}
\DeclareMathOperator{\GL}  {GL}
\DeclareMathOperator{\Hoperator}{H}
\DeclareMathOperator{\Hol} {Hol}
\DeclareMathOperator{\id}  {id}
\DeclareMathOperator{\Imoperator}{Im}
\DeclareMathOperator{\Ooperator}{O}
\DeclareMathOperator{\rank}{rank}
\DeclareMathOperator{\Ric} {Ric}
\DeclareMathOperator{\SL}  {SL}
\DeclareMathOperator{\SO}  {SO}
\DeclareMathOperator{\Sp}  {Sp}
\DeclareMathOperator{\Spin}{Spin}
\DeclareMathOperator{\Stab}{Stab}
\DeclareMathOperator{\SU}  {SU}
\DeclareMathOperator{\tf}  {tf}
\DeclareMathOperator{\tr}  {tr}
\DeclareMathOperator{\U}   {U}
\DeclareMathOperator{\vol} {vol}

\newcommand{\ImbbA}  {\Imoperator                   \bbA  }
\newcommand{\ImwtbbO}{\Imoperator \smash{\widetilde{\bbO}}}

\title[Highly symmetric $2$-plane fields on $5$-manifolds]{Highly symmetric $2$-plane fields on $5$-manifolds \\ and $5$-dimensional Heisenberg group holonomy}

\author{Travis~Willse}

\begin{document}

\address{
  Mathematical Sciences Institute\\
  Building 27\\
  Australian National University\\
  ACT 2601\\
  Australia
}

\email{travis.willse@anu.edu.au}


\begin{abstract}
Nurowski showed that any generic $2$-plane field $D$ on a $5$-manifold $M$ determines a natural conformal structure $c_D$ on $M$; these conformal structures are exactly those (on oriented $M$) whose normal conformal holonomy is contained in the (split, real) simple Lie group $\G_2$. Graham and Willse showed that for real-analytic $D$ the same holds for the holonomy of the real-analytic Fefferman-Graham ambient metric of $c_D$, and that both holonomy groups are equal to $\G_2$ for almost all $D$. We investigate here independently interesting $2$-plane fields for which the associated holonomy groups are a proper subgroup of $\G_2$.

Cartan solved the local equivalence problem for $2$-plane fields $D$ and constructed the fundamental curvature tensor $A$ for these objects. He furthermore claimed to describe locally all $D$ whose infinitesimal symmetry algebra has rank at least $6$ and gave a local quasi-normal form, depending on a single function of one variable, for those that furthermore satisfy a natural degeneracy condition on $A$, but Doubrov and Govorov recently rediscovered a counterexample to Cartan's claim. We show that for all $D$ given by Cartan's alleged quasi-normal form, the conformal structures $c_D$ induced via Nurowski's construction are almost Einstein, that we can write their ambient metrics explicitly, and that the holonomy groups associated to $c_D$ are always the $5$-dimensional Heisenberg group, which here acts indecomposably but not irreducibly. (Not all of these properties hold, however, for Doubrov and Govorov's counterexample.) We also show that the similar results hold for the related class of $2$-plane fields defined on suitable jet spaces by ordinary differential equations $z'(x) = F(y''(x))$ satisfying a simple genericity condition.
\end{abstract}

\dedicatory{This article is dedicated to Mike Eastwood on the occasion of his 60th birthday.}

\maketitle
\pagestyle{myheadings}

\tableofcontents

\section{Introduction}

In a well-known but technically demanding 1910 paper, \cite{Cartan1910}, Cartan solved the local equivalence problem for what in modern geometric language are called $2$-plane fields on $5$-manifolds, and the most interesting such fields are those that satisfy a simple genericity condition. This class is the lowest-dimensional example of $k$-plane fields on $n$-manifolds that admit nontrivial local invariants, but already the geometry of these fields is surprisingly rich and furthermore enjoys close connections with some exceptional geometric objects, including the algebra of the split octonions and the exceptional Lie group $\G_2$.

One of the most striking realizations of these connections was described by Nurowski \cite{Nurowski2005, Nurowski2008} and Leistner and Nurowski \cite{LeistnerNurowski2010}, whose work exploits the geometry of generic $2$-plane fields $D$ on $5$-manifolds $M$ to produce metrics of holonomy equal to $\G_2$ (here and henceforth, $\G_2$ denotes the split real form of the exceptional Lie group). They produce candidate metrics of this kind by concatenating two constructions: First, Nurowski exploited Cartan's solution of the local equivalence problem for these $2$-plane fields to show that any such field $D$ induces a canonical conformal structure $c_D$ of signature $(2, 3)$ on the underlying manifold \cite{Nurowski2005}. Second, the Fefferman-Graham ambient construction associates to any conformal structure $(M, c)$ of signature $(p, q)$ an essentially unique metric $\wtg$ of signature $(p + 1, q + 1)$ on a suitable open subset $\wtM \subseteq \bbR_+ \times M \times \bbR$, though for most $c$ the metric $\wtg$ cannot be identified explicitly \cite{FeffermanGraham2011}. Applying this latter construction to a conformal structure $c_D$ produces a pseudo-Riemannian metric of signature $(3, 4)$, and Leistner and Nurowski produced an explicit family of $2$-plane fields $D$ parametrized by $\bbR^8$ and found corresponding (polynomial) ambient metrics $\wtg_D$ of $c_D$. By giving explicitly a certain object parallel with respect to $\wtg_D$---namely, a $3$-form of a certain algebraic type---they showed that the holonomy groups $\Hol(\wtg_D)$ of the metrics in this family all are contained in the stabilizer in $\SO(3, 4)$ of the $3$-form, which turns out to be $\G_2$, and moreover that for an explicit, dense open subset of parameter values, $\Hol(\wtg_D) = \G_2$ \cite{LeistnerNurowski2010}. This is interesting in part because there are relatively few examples of metrics with this holonomy group.  Later, Graham and Willse showed that for all real-analytic $D$ on oriented $5$-manifolds, there is an ambient metric $\wtg_D$ such that $\Hol(\wtg_D) \leq \G_2$ and that, in a suitable sense, equality holds generically \cite{GrahamWillse2012G2}.

In this article we give an explicit infinite-dimensional family of $2$-plane fields $D$ and corresponding explicit ambient metrics $\wtg_D$ for which the containment of holonomy in $\G_2$ is proper. The $2$-plane fields in this family satisfy two strong invariant criteria (but, pace Cartan's claims, the family is not characterized by these conditions). First, Cartan described the fundamental curvature quantity of $2$-plane fields $D$, which we may interpret as a tensor field $A \in \Gamma(\odot^4 D^*)$ \cite{Cartan1910}. If we complexify $A$, then for each $u \in M$, we may regard the roots of $A_u \ot \bbC \in \odot^4 D^*_u \ot \bbC$ as elements of the complex projective line $\bbP(D_u \ot \bbC)$, and if $A_u \neq 0$, we call the partition of $4$ given by the root multiplicities the root type of $D$ at $u$; for example, we say that $D$ has root type $[4]$ at $u$ if $A_u \ot \bbC$ is nonzero and has a quadruple root, or equivalently, if the line it spans is contained in the rational normal curve in $\odot^4 D^*_u \ot \bbC$. Second, the (infinitesimal) symmetry algebra of $D$ is the Lie algebra $\mfaut(D)$ of vector fields that preserve $D$. Cartan claimed that one can locally encode any $2$-plane field $D$ such that (A) $D$ has constant root type $[4]$ (that is, root type $[4]$ at every $u \in M$), and (B) $\dim \mfaut(D) \geq 6$, in a principal bundle $E \to M$ with $2$-dimensional structure group and a coframe \eqref{equation:structure-equations} on $E$ that depends only a single function $I$, but Doubrov and Govorov have recently produced an interesting counterexample to this claim. For any smooth function $I$ the structure equations of this coframe determines a $2$-plane field $D_I$, and we produce explicit ambient metrics $\wtg_I$ \eqref{equation:I(x)-ambient-metric} on spaces $\wtM_I$ for the conformal structures $c_I := c_{D_I}$ they induce. These conformal structures all enjoy additional special structures, including (exactly) a $2$-dimensional vector space of almost Einstein scales (in fact, almost Ricci-flat scales). Each of these scales in turn corresponds to a parallel null vector field on $\wtM_I$, and so $\Hol(\wtg_I)$ must be contained inside the common stabilizer of those vector fields; we show that $\Hol(\wtg_I)$ is actually this full group, which roughly indicates that the only objects parallel with respect to $\wtg_I$ are those arising from the parallel $\G_2$-structure and the described null vector fields.

We also compute for the conformal structures $c_I$ a closely related notion of holonomy that has recently enjoyed heightened attention \cite{Leistner2006, HammerlSagerschnig2009, CapGoverGrahamHammerl2011}. One can encode any $n$-dimensional conformal manifold $(M, c)$ in a rank-$(n + 2)$ bundle $\mcT \to M$ called the tractor bundle, together with some auxiliary data including a canonical (normal) conformal connection $\nabla^{\mcT}$ on $\mcT$; we call the holonomy $\Hol(\nabla^{\mcT})$ of this connection the (normal) conformal holonomy of $c$.

Partly via explicit computation we prove the following:

\begin{theoremA}
Let $I: U \to \bbR$ be a smooth function whose domain $U \subset \bbR$ is open and connected, and let $\wtg_I$ be the Ricci-flat ambient metric \eqref{equation:I(x)-ambient-metric} of $c_I$. Then,
\[
    \Hol(\nabla^{\mcT}_I) \cong \Hol(\wtg_I) \cong \Hoperator_5 \textrm{,}
\]
where $\nabla^{\mcT}_I$ is the tractor connection of $c_I$ and $\Hoperator_5$ is the $5$-dimensional Heisenberg group.
\end{theoremA}

The local holonomy $\Hol^*_u(\nabla^E)$ of a connection $\nabla^E$ on a vector bundle is defined in the paragraphs after Definition \ref{definition:holonomy} below.

Some algebra shows that these holonomy groups act indecomposably, and so the ambient metrics $\wtg_I$ respectively provide new examples of metric holonomy groups that act indecomposably but not irreducibly; among metric connections this phenomenon can occur only in indefinite signature.

The remainder of this article is organized as follows: Section \ref{section:generic-2-plane-fields} collects some general facts about generic $2$-plane fields $D$ on $5$-manifolds, including about the Cartan curvature tensor $A$. We construct a homogeneous model of such $2$-plane fields using just the algebraic structure of the split octonions, $\wtbbO$; this connects the geometry of these fields to the exceptional Lie group $\G_2$, which can be realized as the automorphism group of $\wtbbO$, and we recall some specific algebraic facts about $\G_2$ relevant to the geometry of the $2$-plane fields $D_I$. We also recall Goursat's quasi-normal form, which locally realizes any $D$ as a differential equation $z' = F(x, y, y', y'', z)$ for some function $F(x, y, p, q, z)$ defined on a neighborhood of the origin in $\bbR^5$ and for which $F_{qq}$ is nowhere zero; we exploit this form in direct computations in Section \ref{section:special-classes}. Section \ref{section:metric-conformal-geometry} gives some basic facts about conformal geometry, including the construction of the conformal tractor bundle and ambient metric, and also recalls and relates various notions of holonomy, including the normal holonomy $\Hol(\nabla^{\mcT})$ of a conformal structure. Section \ref{section:Nurowski-conformal-structures} discusses Nurowski's construction of the conformal structure $c_D$ from a generic $2$-plane field $D$ on a $5$-manifold and describes briefly a manifestly invariant construction of that structure in the language of parabolic geometry, for which we deliberately provide no other background; the standard reference for this topic is \cite{CapSlovak2009}. We give several key results about these conformal structures, including a characterization of them among all conformal structures due to Hammerl and Sagerschnig \cite{HammerlSagerschnig2009} and some facts about holonomy groups associated to them. Finally, in Section \ref{section:special-classes} we give explicit data for the $2$-plane fields $D_I$, including formulas for Ricci-flat ambient metrics $\wtg$ and the parallel objects on the tractor bundle and ambient manifold that guarantee the containment of the indicated holonomy groups in $\Hoperator_5$. We use these data to prove Theorem A, and various other results related to these $2$-plane fields, including that the holonomy of any Ricci-flat representative of $c_I$ is equal to $\bbR^3$. We then consider a class of $2$-plane fields, namely the fields $D_{F(q)}$ for which the function $F$ in the local normal form depends only on $q$, which thus correspond to differential equations $z' = F(y'')$. This class is closely related to that of the $2$-plane fields $D_I$, and we show that the holonomy results we prove for the $2$-plane fields $D_I$ essentially hold for the $2$-plane fields $D_{F(q)}$ too. Finally, we discuss briefly Doubrov and Govorov's striking counterexample, which behaves substantially differently from the $2$-plane fields in the above, but for which we postpone detailed discussion to a later paper.

All objects are smooth by hypothesis except where stated otherwise.

Many computations in this work were done with Maple, and in particular the package \texttt{DifferentialGeometry}. It is a pleasure to thank Ian Anderson, that package's primary author, for hosting the author at Utah State University in March 2012 after the conference ``Differential Geometry of Distributions,'' when some of this work was done, as well as for his assistance with the package. The author thanks Robert Bryant for suggesting to the author Cartan's 1910 paper as a possible source of $2$-plane fields with interesting associated holonomy groups at the October 2010 meeting of the Pacific Northwest Geometry Seminar at the University of Oregon. The author also thanks Mike Eastwood, Ravi Shroff, Dennis The, and the referee for various helpful comments made during the paper's preparation and revision, and Boris Doubrov for a helpful exchange regarding the counterexample to Cartan's classification he produced with Artem Govorov. Support from the Australian Research Council is gratefully acknowledged.

\section{Generic $2$-plane fields on $5$-manifolds}\label{section:generic-2-plane-fields}

\begin{definition}
A $2$-plane field $D$ on a $5$-manifold $M$ is \textbf{generic} if $[D, [D, D]] = TM$.
\end{definition}

If $D$ is generic, then the \textbf{derived plane field} $[D, D]$ has constant rank $3$.

Two generic $2$-plane fields $(M, D)$ and $(\whM, \whD)$ are \textbf{equivalent} if there is a diffeomorphism $\varphi: M \to \whM$ such that $T_x \varphi \cdot D = \whD_{\varphi(x)}$ for all $x \in M$, and they are merely \textbf{locally equivalent} at $x \in M$ and $\whx \in \whM$ if there are neighborhoods $V$ of $x$ and $\whV$ of $\whx$ such that the restricted $2$-plane fields $(V, D\vert_V)$ and $(\whV, \whD\vert_{\whV})$ are equivalent. If $(M, D)$ and $(\whM, \whD)$ are locally equivalent at $x$ and $\whx$ for all $x \in M$ and $\whx \in \whM$, we simply say that the $2$-plane fields are locally equivalent, without reference to a choice of points. When using any of the above notions of equivalence, we may suppress mention of the underlying manifolds when their identities are clear from context.

\begin{definition}
Let $E$ be a $k$-plane field on a smooth manifold $M$. A vector field $\xi \in \Gamma(TM)$ is an \textbf{infinitesimal symmetry} of $E$ if $\mcL_{\xi} \eta \in \Gamma(E)$ for all $\eta \in \Gamma(E)$. The \textbf{(infinitesimal) symmetry algebra} of $E$ is the space $\mfaut(E)$ of infinitesimal symmetries of $E$, and the compatibility of the Lie bracket of vector fields with the Lie derivative ensures that $\mfaut(E)$ is a Lie subalgebra of $\Gamma(TM)$ under that operation.
\end{definition}

\subsection{The Cartan curvature tensor}\label{subsection:Cartan-curvature-tensor}

Cartan constructed the fundamental curvature invariant, which we call the \textbf{Cartan curvature (tensor)}, of a generic $2$-plane field $D$ on a $5$-manifold $(M, D)$: a symmetric quartic form on $D$, that is, an element $\smash{A \in \Gamma(\odot^4 D^*)}$. This form is precisely the harmonic curvature of the type of parabolic geometry that encodes this structure, and it vanishes identically iff $D$ is locally equivalent to the homogeneous model of this geometry described in the next subsection.

By complexifying and projectivizing, at each point $x \in M$ we may regard the roots of $\smash{A_x \otimes \bbC \in \odot^4 D_x^* \ot \bbC}$ as elements of the complex projective line $\bbP(D_x \otimes \bbC)$. If $A_x \neq 0$, then $A_x \ot \bbC$ has exactly four roots counting multiplicity, and we call the partition $\Lambda$ of $4$ given by the multiplicities of the roots the \textbf{root type} of $D$ at $x$. By convention, if $A_x = 0$, we say that $D$ has root type $[\infty]$ at $x$. If $D$ has a given root type $\Lambda$ at all $x \in M$, we just say that $D$ has (constant) root type $\Lambda$.

\subsection{The split octonions and the homogeneous model}\label{subsection:split-octonions-homogeneous-model}

A natural model for generic $2$-plane fields on $5$-manifolds can be efficiently and beautifully realized using an $8$-dimensional real algebra called the split octonions, which in turn is intimately related to several exceptional objects.

Up to isomorphism there are exactly seven \textbf{composition algebras} over $\bbR$, that is, algebras over $\bbR$ with a unit and a nondegenerate bilinear form $\ab{ \cdot , \cdot}$ that satisfies $\ab{xy, xy} = \ab{x, x} \ab{y, y}$ for all elements $x$ and $y$ in the algebra; the facts here about composition algebras are given, for example, in \cite{Harvey1990}, where they are called normed algebras. Four of these are the celebrated normed division algebras, $\bbR$, $\bbC$, $\bbH$, and $\bbO$. The remaining three are the so-called split analogues of the latter three; the largest and richest of these, both algebraically and geometrically, is the \textit{split octonions}, which we denote $\wtbbO$. This algebra has dimension $8$ over $\bbR$, and its bilinear form has signature $(4, 4)$.

We call an element $a$ of a composition algebra $\bbA$ \textbf{imaginary} if $\ab{1, a} = 0$, and we denote the set of such elements by $\ImbbA$. Restricting $\ab{ \cdot, \cdot}$ to $\ImbbA$ defines a nondegenerate bilinear form (which we again denote by $\ab{ \cdot , \cdot}$); we can then define the \textbf{cross product}
\[
    \cdot \times \cdot : \ImbbA \times \ImbbA \to \ImbbA
\]
on that subspace by
\[
    x \times y := -\Imoperator (x y) \textrm{.}
\]
Regarding $\times$ as a $(2, 1)$-tensor on $\ImbbA$ and dualizing (and changing signs, to agree with the convention in \cite{Sagerschnig2006}) defines a $3$-tensor $\Phi \in \otimes^3 \Lambda^3 (\ImbbA)^*$ by
\[
    \Phi(x, y, z) := - \ab{ x \times y , z } = \ab{ x y , z } \textrm{,}
\]
and one can show that it is totally antisymmetric.

We henceforth restrict our attention to $\bbA = \wtbbO$; for further details of most of the constructions in the rest of the subsection and Subsection \ref{subsection:G2}, see \cite{Sagerschnig2006} and \cite{HammerlSagerschnig2009}. Since bilinear form of this algebra has signature $(4, 4)$, the induced bilinear form on $\ImwtbbO$ has signature $(3, 4)$. We can recover all of the above structure on $\wtbbO$ and hence $\ImwtbbO$ from $\Phi$ alone; in particular, the bilinear form on $\ImwtbbO$ satisfies
\begin{equation}\label{equation:cross-product-to-bilinear-form}
    \ab{ x , y } = -\tfrac{1}{6} \tr (z \mapsto x \times (y \times z)) \textrm{.}
\end{equation}
The algebra and geometry of $\wtbbO$ are in some ways richer than that of its analogue, $\bbO$, because the former contains zero divisors; these are exactly the nonzero vectors null with respect to the bilinear form. To exploit this structure, we define the \textbf{(punctured) null cone} to be the set
\[
    \mcN := \set{x \in \ImwtbbO - \set{0} : \ab{x, x} = 0}
\]
of nonzero null vectors in $\ImwtbbO$ and define the \textbf{(null) quadric} to be its projectivization, $\mcQ := \bbP(\mcN)$ (by construction, $\mcQ$ is diffeomorphic to $(\bbS^2 \times \bbS^3) / \bbZ_2$, where the nonidentity element of $\bbZ_2$ acts by the antipodal map simultaneously on the two spheres). For any $x \in \ImwtbbO$ we define the \textbf{(algebraic) annihilator} of $x$ to be the vector space
\[
    \Ann x := \set{y \in \ImwtbbO : x \times y = 0} = \set{y \in \ImwtbbO : \Phi(x, y, \cdot) = 0} \textrm{,}
\]
and one can show that $\dim \Ann x = 3$ if $x$ is nonzero and null and $\Ann x = [x]$ if $x$ is non-null, where $[x]$ denotes the span of $x$.

For $x \in \mcN$, some easy algebra yields the inclusions
\[
    [x] \subset \Ann x \subset (\Ann x)^{\perp} \subset [x]^{\perp} \textrm{,}
\]
which we can together regard as a vector space filtration of $\smash{T_x \mcN = [x]^{\perp}}$. Varying $x$ defines a filtration of the tangent bundle $T\mcN$ by plane fields of constant rank $1$, $3$, $4$, and $6$. This descends to a filtration of the tangent bundle of the quadric:
\[
    \set{0} = D^0 \subset D^{-1} \subset D^{-2} \subset D^{-3} = T\mcQ \textrm{.}
\]
In particular, $\rank D^{-1} = 2$ and $\rank D^{-2} = 3$. We denote $\Delta := D^{-1}$; computing shows that $[\Delta, \Delta] = D^{-2}$ and $[\Delta, [\Delta, \Delta]] = [D^{-1}, D^{-2}] = D^{-3} = T\mcQ$. (Given constant-rank plane fields $V$ and $W$, the set $[V, W]$, which a priori need not have constant rank, is
\[
    [V, W]_p = \set{[\xi, \eta]_p : \xi \in \Gamma(V), \eta \in \Gamma(W)} \textrm{.} )
\]
In particular, $\Delta$ is a generic $2$-plane field on $\mcQ$, and we call the pair the \textbf{homogeneous model} of the geometry of such fields. We say that a generic $2$-plane field on a $5$-manifold is \textbf{locally flat} if it is locally equivalent to $(\mcQ, \Delta)$, and one can show that local flatness is equivalent to the Cartan curvature tensor $A$ being identically zero \cite{Cartan1910}.

\subsection{$\G_2$}\label{subsection:G2}

The Lie algebra of the automorphism group $\Aut \wtbbO$ of the split octonions is simple, has dimension 14, and has indefinite Killing form, so it is the split real form of the simple complex Lie algebra of type $\G_2$. We thus henceforth denote this automorphism group by $\G_2$.

Since $\G_2$ preserves $\set{1}$ and $\ab{\cdot, \cdot}$, it also preserves $\smash{[1]^{\perp} = \ImwtbbO}$ and $\Phi \in \Lambda^3 (\ImwtbbO)^*$. One can show that $\G_2$ is connected; then, since \eqref{equation:cross-product-to-bilinear-form} realizes the bilinear form $\ab{\cdot, \cdot}$ on $\ImwtbbO$ in terms of its algebraic structure, $\G_2$ also preserves that form and hence admits a natural embedding $\G_2 \hookrightarrow \SO(3, 4)$. The automorphism action of $\G_2$ thus restricts to an action on the null cone $\mcN$, and by linearity it descends to an action on the quadric $\mcQ$.

Conversely, the full algebra $\wtbbO$ can be recovered from $\Phi$, and so $\G_2$ is precisely the stabilizer of $\Phi$ under the induced action of $\GL(\ImwtbbO)$ on $\Lambda^3 (\ImwtbbO)^*$. The stabilizer of any other $3$-form in the orbit $\GL(\ImwtbbO) \cdot \Phi$ is a conjugate of $\G_2$ in $\GL(\ImwtbbO)$, and we call the $3$-forms in this orbit \textbf{split-generic}. (The modifier \textbf{generic} indicates that this orbit turns out to be open. In fact, the action of $\GL(\ImwtbbO)$ on $\Lambda^3 (\ImwtbbO)^*$ has exactly two open orbits. The stabilizer of any element in the open orbit that does not contain $\Phi$ is just the compact real form of the complexification $\G_2^{\bbC}$ of $\G_2$, and this compact form can be realized as the algebra automorphism group of the octonions, $\bbO$.)

Given a $7$-dimensional real vector space $\bbV$ and any vector space isomorphism $\tau: \ImwtbbO \to \bbV$, we say that a $3$-form $\Phi \in \Lambda^3 \bbV^*$ is split-generic if and only if $\tau^* \Phi$ is; by construction this characterization does not depend on the choice of $\tau$. So, up to algebra isomorphism we may realize the imaginary split octonions, endowed with the cross product, by giving such a pair $(\bbV, \Phi)$ with $\Phi$ split-generic.

Now, given such a realization $(\bbV, \Phi)$ of the imaginary split octonions, one can realize the induced inner product $\ab{\cdot, \cdot}$ explicitly in terms of the $3$-form $\Phi$. Any $3$-form $\varphi$ on a $7$-dimensional real vector space $\bbV$ induces a symmetric $\Lambda^7 \bbV^*$-valued bilinear form, namely, $(X, Y) \mapsto (X \imult \varphi) \wedge (Y \imult \varphi) \wedge \varphi$. One can show that this form is nondegenerate if and only if $\varphi$ is generic, in which case it distinguishes a nonzero volume form $\vol \in \Lambda^7 \bbV^*$ \cite{Bryant1987, Hitchin2001} and hence yields an $\bbR$-valued bilinear form on $\bbV$: If we regard the bilinear form as a linear map $\otimes^2 \bbV \to \Lambda^7 \bbV^*$, dualizing gives a map $\bbV \to \bbV^* \ot \Lambda^7 \bbV^*$. Its determinant is a map $\det: \Lambda^7 \bbV \to \Lambda^7 (\bbV^* \ot \Lambda^7 \bbV^*) \cong \otimes^8 \Lambda^7 \bbV^*$, and dualizing again gives a map $\det: \bbR \to \otimes^9 \Lambda^7 \bbV^*$ which is nonzero because the bilinear form is nondegenerate, and which we may regard as a distinguished element of $\otimes^9 \Lambda^7 \bbV^*$. Since $\bbV$ is real, there is a unique element $\vol \in \Lambda^7 \bbV^*$ such that $\vol^{\ot 9} = \det$. So, if $\varphi$ is generic, we can define a symmetric bilinear form $\smash{H(\varphi) \in \odot^2 \bbV^*}$ by
\begin{equation}\label{equation:induced-bilinear-form}
    H(\varphi)(X, Y) \vol := \sqrt{6} (X \imult \varphi) \wedge (Y \imult \varphi) \wedge \varphi \textrm{.}
\end{equation}
The factor $\sqrt{6}$ is chosen so that $\vol$ coincides with the volume form induced by the bilinear form for the orientation $\vol$ determines on $\bbV$. The form $H(\varphi)$ is split-generic if and only if the bilinear form has signature $(3, 4)$, and it is generic but not split-generic if and only if the form has signature $(7, 0)$.

For our purposes it will be useful to have a concrete realization of the split octonions, and we borrow a computationally convenient choice from \cite{HammerlSagerschnig2009}: Set $\bbV = \bbR^7$, let $(E_a)$ denote the standard basis on $\bbV$ and $(e^a)$ its dual basis, and define
\begin{multline*}
    \Phi := \tfrac{1}{\sqrt{6}} (-\sqrt{2} e^1 \wedge e^5 \wedge e^6 - e^2 \wedge e^4 \wedge e^5 - e^3 \wedge e^4 \wedge e^6 \\
                + e^1 \wedge e^4 \wedge e^7 - \sqrt{2} e^2 \wedge e^3 \wedge e^7) \textrm{.}
\end{multline*}
The matrix representation in the basis $(E_a)$ of the bilinear form $H(\Phi)$ is
\[
    \left(
        \begin{array}{ccccc}
            0 &    0   &  0 &      0 & 1 \\
            0 &    0   &  0 & \bbI_2 & 0 \\
            0 &    0   & -1 &    0   & 0 \\
            0 & \bbI_2 &  0 &    0   & 0 \\
            1 &    0   &  0 &    0   & 0 \\
        \end{array}
    \right) \textrm{,}
\]
where $\bbI_2$ is the $2 \times 2$ identity matrix. In particular, $H(\Phi)$ nondegenerate and has signature $(3, 4)$, so $\Phi$ is split-generic; we denote the bilinear form by $\ab{X, Y} := H(\Phi)(X, Y)$.

The Lie algebra $\mfg_2$ is the algebra of derivations of the split octonions, which by the discussion earlier in the subsection is just the annihilator of $\Phi$ in the Lie algebra $\mfgl(\bbV)$ of derivations of $\bbV$. Computing gives that, in the basis $(E_a)$, $\mfg_2$ has matrix representation
\begin{equation}\label{equation:g2-representation}
    \left\{
        \left(
            \begin{array}{ccccc}
                \tr A &                 Z        &               s   &                  W      &      0   \\
                    X &                 A        & \sqrt{2} \bbJ Z^T & \frac{s}{\sqrt{2}} \bbJ & -    W^T \\
                    r & -      \sqrt{2} X^T \bbJ &               0   & -\sqrt{2}        Z \bbJ &      s   \\
                    Y & -\frac{r}{\sqrt{2}} \bbJ & \sqrt{2} \bbJ X   & -                A^T    & -    Z^T \\
                    0 & -               Y^T      &               r   & -                X^T    & -\tr A   \\
            \end{array}
        \right)
        : \begin{array}{c} A \in \mfgl(2, \bbR) \\ X, Y \in \bbR^2 \\ Z, W \in (\bbR^2)^* \\ r, s \in \bbR \\ \end{array}
        \right\} \textrm{,}
\end{equation}
where
\[
    \bbJ =
        \left(
            \begin{array}{cc}
                0 & -1 \\
                1 &  0 \\
            \end{array}
        \right) \textrm{.}
\]
One can use \eqref{equation:g2-representation} to show that $\G_2$ acts transitively on the null cone $\mcN$ (and hence the null quadric $\mcQ)$, so we may regard $\mcQ$ as the homogeneous space $\G_2 / P_1$, where $P_1$ is the stabilizer in $\G_2$ of a point in $\mcQ$, that is, of a null line in $\ImwtbbO$. Furthermore, the $2$-plane field $\Delta$ on $\mcQ$ was constructed algebraically from $\Phi$, so it is invariant under this action, and this motivates the moniker \textit{homogeneous model} for $(\mcQ, \Delta)$.

We will be interested in determining the holonomy of certain $7$-dimensional metrics (see Section \ref{section:special-classes}), all of which admit a parallel split-generic $3$-form and two linearly independent null vector fields. (We say that a $3$-form on a smooth $7$-manifold is split-generic if and only if its value at each point is split-generic; if a parallel $3$-form is split-generic at one point, it is split-generic everywhere.) These holonomy groups will be contained in the common stabilizer in $\SO(H(\Phi)) \cong \SO(3, 4)$ of the split-generic $3$-form $\Phi$ and two linearly independent null vectors $x$ and $y$; however, this stabilizer depends on the relative configuration of those three tensors.

\begin{proposition}\label{proposition:stabilizer-two-null-vectors}
Suppose $x, y \in \mcN$, and let $\Phi$ be the split-generic $3$-form that defines the algebraic structure on $\bbV$. The isomorphism type of the common stabilizer $\Stab_{\SO(3, 4)} (\Phi) \cap \Stab_{\SO(3, 4)} (x) \cap \Stab_{\SO(3, 4)} (y)$ is
\[
    \left\{
    \begin{array}{ll}
                 K  , & \textrm{if $[x] = [y]$} \\
        \Hoperator_5, & \textrm{if $\Phi(x, y, \, \cdot \,) = 0 , \, [x] \neq [y]$} \\
              \bbR^3, & \textrm{if $\ab{x, y} = 0, \, \Phi(x, y, \, \cdot \,) \neq 0$} \\
        \SL(2, \bbR), & \textrm{if $\ab{x, y} \neq 0$}
    \end{array}
    \right. \textrm{,}
\]
where $[z]$ denotes the line in $\ImwtbbO$ spanned by the vector $z$, $K$ is the stabilizer of an (arbitrary) vector in $\mcN$, and $\Hoperator_5$ is the $5$-dimensional Heisenberg group.
\end{proposition}
\begin{proof}
By the discussion earlier in this subsection, the stabilizer of $\Phi$ in $\SO(\bbV)$ is $\G_2$, and so the common stabilizer is just the common stabilizer in $\G_2$ of $x$ and $y$.

Since $\G_2$ acts transitively on $\mcN$, we may identify $x$ with the null vector $e_1 \in \bbV$; then, where $K$ is the stabilizer in $\G_2$ of this vector, the common stabilizer is just $\Stab_K (y)$.

Consulting \eqref{equation:g2-representation} shows that the Lie algebra of $K$ is
\begin{equation}\label{equation:k-representation}
    \mfk =
    \left\{
        \left(
            \begin{array}{ccccc}
                0 & Z &               s   &                  W      &      0   \\
                0 & A & \sqrt{2} \bbJ Z^T & \frac{s}{\sqrt{2}} \bbJ & -    W^T \\
                0 & 0 &               0   &      -\sqrt{2}   Z \bbJ &      s   \\
                0 & 0 &               0   &      -           A^T    & -    Z^T \\
                0 & 0 &               0   &                  0      &      0   \\
            \end{array}
        \right)
        : \begin{array}{c} A \in \mfsl(2, \bbR) \\ Z, W \in (\bbR^2)^* \\ s \in \bbR \\ \end{array}
        \right\} \textrm{.}
\end{equation}

Using this realization it is easy to check that the action of $K$ partitions $\mcN$ into the following orbits; we also give a representative of each orbit in terms of the basis $(e_a)$:
\begin{itemize}
    \item singletons $\set{\lambda x}$, $\lambda \in \bbR^*$; $\lambda e_1$
    \item the set $\Ann x - [x]$; $e_2$
    \item the set $\smash{([x]^{\perp} - \Ann x) \cap \mcN}$; $e_5$
    \item hypersurfaces $\set{ z \in \mcN : \ab{x, z} = \lambda}$, $\lambda \in \bbR^*$; $\lambda e_7$.
\end{itemize}
Consulting \eqref{equation:k-representation} gives that the common stabilizer in $\mfk$ of $e_1$ and $e_2$ is
\begin{equation}\label{equation:h5-representation}
    \left\{
        \left(
            \begin{array}{ccccccc}
                0 & 0 & Z_2    &           s   &                     W_1    &                      W_2 &  0    \\
                0 & 0 & a_{12} & -\sqrt{2} Z_2 &                     0      & -\tfrac{1}{\sqrt{2}} s   & -W_1  \\
                0 & 0 & 0      &           0   & \tfrac{1}{\sqrt{2}} s      &                      0   & -W_2  \\
                0 & 0 & 0      &           0   &           -\sqrt{2} Z_2    &                      0   &  s    \\
                0 & 0 & 0      &           0   &                     0      &                      0   &  0    \\
                0 & 0 & 0      &           0   &                     a_{12} &                      0   &  -Z_2 \\
                0 & 0 & 0      &           0   &                     0      &                      0   &  0    \\
            \end{array}
        \right)
        \right\}
        \cong
        \mfh_5
        \textrm{,}
\end{equation}
where $a_{12}, Z_2, s, W_1, W_2 \in \bbR$; this algebra is isomorphic to the $5$-dimensional Heisenberg Lie algebra. This and analogous computations for the three other orbit types give that the isomorphism types of the stabilizers in $K$ of the given orbit representatives (and hence that of arbitrary elements $y$ in their respective orbits) are, respectively, $K$, $\Hoperator_5$, $\bbR^3$, and $\SL(2, \bbR)$.
\end{proof}

\begin{remark}
For each of the four cases in the proposition, one can read some of the information about the above stabilizers directly from the root diagram of $\mfg_2$. By choosing an appropriate Cartan subalgebra, we may identify the stabilizer Lie algebra $\mfk = \mfstab_{\mfg_2} (x)$ as the subalgebra spanned by the root spaces of the roots in the indicated box in the diagram below together with the $1$-dimensional subalgebra of the Cartan subalgebra (which corresponds to the center node) that fixes $x$. In fact, if we identify $x$ with $e_1$ as in the proof, the Cartan subalgebra of $\mfg_2$ \eqref{equation:g2-representation} comprising the diagonal matrices will do.

For any $y \in \Ann x - [x]$, one can choose a basis of $\bbV$ so that, for example, $x = e_1$ (again) and $y = e_2$, and if we again take the Cartan subalgebra to be the set of diagonal matrices, the stabilizer $\smash{\whmfk := \mfstab_{\mfg_2} (y)}$ of $y$ in $\mfg_2$ is the one indicated. The common stabilizer algebra $\mfstab_{\mfg_2} (x) \cap \mfstab_{\mfg_2} (y)$ is just the span of the root spaces in both $\mfk$ and $\smash{\whmfk}$, namely those of the circled roots (in particular, no nonzero element of the Cartan subalgebra stabilizes both $e_1$ and $e_2$); the diagram shows that this algebra is isomorphic to $\mfh_5$.

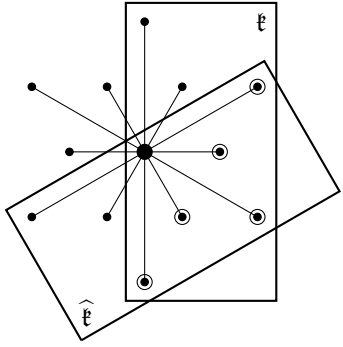
\begin{center}
\begin{tikzpicture}[scale=1]
\filldraw ( 0  ,  0    ) circle (0.1 );
\draw ( 0  , -1.732) -- (0  ,  1.732);
\filldraw ( 0  , -1.732) circle (0.05);
    \draw ( 0  , -1.732) circle (0.1 );
\filldraw ( 0  ,  1.732) circle (0.05);
\draw (-0.5, -0.866) -- (0.5,  0.866);
\filldraw (-0.5, -0.866) circle (0.05);
\filldraw ( 0.5,  0.866) circle (0.05);
\draw (-1.5, -0.866) -- (1.5,  0.866);
\filldraw (-1.5, -0.866) circle (0.05);
\filldraw ( 1.5,  0.866) circle (0.05);
    \draw ( 1.5,  0.866) circle (0.1 );
\draw (-1  ,  0    ) -- (1  ,  0    );
\filldraw (-1  , -0    ) circle (0.05);
\filldraw ( 1  ,  0    ) circle (0.05);
    \draw ( 1  ,  0    ) circle (0.1 );
\draw (-1.5,  0.866) -- (1.5, -0.866);
\filldraw (-1.5,  0.866) circle (0.05);
\filldraw ( 1.5, -0.866) circle (0.05);
    \draw ( 1.5, -0.866) circle (0.1 );
\draw (-0.5,  0.866) -- (0.5, -0.866);
\filldraw (-0.5,  0.866) circle (0.05);
\filldraw ( 0.5, -0.866) circle (0.05);
    \draw ( 0.5, -0.866) circle (0.1 );

\draw [thick] (-0.25 ,  -1.982) rectangle ( 1.750,  1.982) node[anchor=north east] {$\mfk$};
\draw [thick] (-0.841,  -2.507) -- (-1.841, -0.774) -- ( 1.591,  1.208) -- ( 2.591, -0.525) -- (-0.841, -2.507) node[anchor=south,xshift=2,yshift=1.5] {$\whmfk$};
\end{tikzpicture}
\end{center}

Proceeding analogously for $y$ in the $K$-orbit $\smash{([x]^{\perp} - \Ann x) \cap \mcN}$ and any $K$-orbit $\set{ z \in \mcN : \ab{x, z} = \lambda}$, $\lambda \in \bbR^*$, respectively yields diagrams
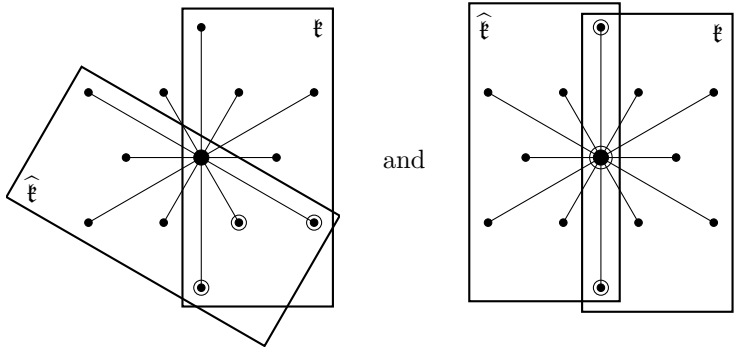
\begin{center}
\begin{tabular}{m{4.5cm}m{1cm}m{4.5cm}m{0.1cm}}
\hfill
\begin{tikzpicture}[scale=1]
\clip (-2.592, -2.508) rectangle (1.842, 2.508);
\filldraw ( 0  ,  0    ) circle (0.1 );
\draw ( 0  , -1.732) -- (0  ,  1.732);
\filldraw ( 0  , -1.732) circle (0.05);
    \draw ( 0  , -1.732) circle (0.1 );
\filldraw ( 0  ,  1.732) circle (0.05);
\draw (-0.5, -0.866) -- (0.5,  0.866);
\filldraw (-0.5, -0.866) circle (0.05);
\filldraw ( 0.5,  0.866) circle (0.05);
\draw (-1.5, -0.866) -- (1.5,  0.866);
\filldraw (-1.5, -0.866) circle (0.05);
\filldraw ( 1.5,  0.866) circle (0.05);
\draw (-1  ,  0    ) -- (1  ,  0    );
\filldraw (-1  , -0    ) circle (0.05);
\filldraw ( 1  ,  0    ) circle (0.05);
\draw (-1.5,  0.866) -- (1.5, -0.866);
\filldraw (-1.5,  0.866) circle (0.05);
\filldraw ( 1.5, -0.866) circle (0.05);
    \draw ( 1.5, -0.866) circle (0.1 );
\draw (-0.5,  0.866) -- (0.5, -0.866);
\filldraw (-0.5,  0.866) circle (0.05);
\filldraw ( 0.5, -0.866) circle (0.05);
    \draw ( 0.5, -0.866) circle (0.1 );

\draw [thick] (-0.25 , -1.982) rectangle ( 1.750,  1.982) node[anchor=north east] {$\mfk$};
\draw [thick] (-2.591, -0.525) -- ( 0.841, -2.507) -- ( 1.841, -0.774) -- (-1.591,  1.208) -- (-2.591, -0.525) node[anchor=west, xshift=4, yshift=3] {$\whmfk$};
\end{tikzpicture}
&
\centering
and
&
\begin{tikzpicture}[scale=1]
\filldraw ( 0  ,  0    ) circle (0.1 );
    \draw ( 0  ,  0    ) circle (0.15);
\draw ( 0  , -1.732) -- (0  ,  1.732);
\filldraw ( 0  , -1.732) circle (0.05);
    \draw ( 0  , -1.732) circle (0.1 );
\filldraw ( 0  ,  1.732) circle (0.05);
    \draw ( 0  ,  1.732) circle (0.1 );
\draw (-0.5, -0.866) -- (0.5,  0.866);
\filldraw (-0.5, -0.866) circle (0.05);
\filldraw ( 0.5,  0.866) circle (0.05);
\draw (-1.5, -0.866) -- (1.5,  0.866);
\filldraw (-1.5, -0.866) circle (0.05);
\filldraw ( 1.5,  0.866) circle (0.05);
\draw (-1  ,  0    ) -- (1  ,  0    );
\filldraw (-1  , -0    ) circle (0.05);
\filldraw ( 1  ,  0    ) circle (0.05);
\draw (-1.5,  0.866) -- (1.5, -0.866);
\filldraw (-1.5,  0.866) circle (0.05);
\filldraw ( 1.5, -0.866) circle (0.05);
\draw (-0.5,  0.866) -- (0.5, -0.866);
\filldraw (-0.5,  0.866) circle (0.05);
\filldraw ( 0.5, -0.866) circle (0.05);

\draw [thick] (-0.25 ,  -2.052) rectangle ( 1.750,  1.912) node[anchor=north east] {$  \mfk$};
\draw [thick] ( 0.25 ,  -1.912) rectangle (-1.750,  2.052) node[anchor=north west] {$\whmfk$};
\end{tikzpicture}
\hfill
&
.
\end{tabular}
\end{center}
In the diagram for $\set{ z \in \mcN : \ab{x, z} = \lambda}$, $\lambda \in \bbR^*$, the circle around the center node indicates that the common stabilizer of $x$ and $y$ contains a $1$-dimensional subalgebra of the Cartan subalgebra of $\mfg_2$.
\end{remark}

\begin{remark}
In the language of \cite{BaezHuerta2012}, the four conditions on $x$ and $y$ in the statement of Proposition \ref{proposition:stabilizer-two-null-vectors} are equivalent to the pair $([x], [y]) \in \mcQ \times \mcQ$ of null lines being $0$, $1$, $2$, and $3$ rolls apart, respectively, so the proposition shows in particular that for each $k \in \set{0, 1, 2, 3}$, $\G_2$ acts transitively on the space of pairs of null lines that are $k$ rolls apart, that is, that these are precisely the four orbits of the induced $\G_2$ action on pairs of null lines.
\end{remark}

\subsection{Ordinary differential equations $z' = F(x, y, y', y'', z)$}\label{subsection:ordinary-differential-equations}

Consider a second-order ordinary differential equation in the Monge normal form
\begin{equation}\label{equation:ODE}
    z' = F(x, y, y', y'', z) \textrm{,}
\end{equation}
where $y$ and $z$ are functions of $x$. Introducing coordinates $p$ and $q$ for the $y'$ and $y''$ identifies the partial jet space $J^{0, 2}(\bbR, \bbR^2)$ with $\bbR^5$ (with coordinates $(x, y, p, q, z)$) and realizes the differential equation \eqref{equation:ODE} as the exterior differential system
\begin{equation}\label{equation:local-coframe-partial}
    \left\{
        \begin{array}{ll}
            \omega^1 := dy - p \,dx \\
            \omega^2 := dz - F(x, y, p, q, z) \,dx - F_q(x, y, p, q, z) (dp - q \, dx) \\
            \omega^3 := dp - q \,dx
        \end{array}
    \right.
\end{equation}
on $\bbR^5$: Explicitly, a triple $(x, y(x), z(x))$ is a solution of \eqref{equation:ODE} if and only if its prolongation $(x, y(x), y'(x), y''(x), z(x))$ is an integral curve of the common kernel of \eqref{equation:local-coframe-partial}.

Since $\set{\omega^1, \omega^2, \omega^3}$ is linearly independent, the common kernel $\ker \set{\omega^1, \omega^2, \omega^3}$ is a $2$-plane field $D_F$ on $\dom F$, and checking directly shows that it is generic if and only if $F_{qq}$ is nowhere zero. These defining $1$-forms were chosen so that the derived $3$-plane field satisfies $[D_F, D_F] = \ker \set{\omega^1, \omega^2}$.

Goursat showed that all generic $2$-plane fields arise this way, at least locally.

\begin{lemma}\cite[\S76]{Goursat1922}
Let $D$ be a generic $2$-plane field on a $5$-manifold $M$, and fix $u \in M$. Then, there is a function $F$ defined on an open subset of $\bbR^5$ containing $0$ such that $D$ and $D_F$ are locally equivalent near $u$ and $0$.
\end{lemma}

For accessible proofs of this lemma, see \cite[p.~2.6]{BryantHsu1993}, which proves a generalization of the lemma to manifolds of dimension $5$ and higher, or see \cite[Theorem 3]{Kruglikov2011} or \cite[\S3.3]{Strazzullo2009}.

There is some redundancy in the choice of $F$ in this lemma---somewhat more precisely, different functions $F$ can yield equivalent $2$-plane fields $D_F$---so we refer to \eqref{equation:local-coframe-partial} only as the \textbf{local Monge (quasi-)normal form} for generic $2$-plane fields on $5$-manifolds. The field $D_{q^2}$ is locally flat \cite{Cartan1910}.

For later use, we augment \eqref{equation:local-coframe-partial} with two auxiliary forms to produce a local coframe $(\omega^a)$ of $TM$:
\begin{equation}\label{equation:local-coframe}
    \left\{
        \begin{array}{l}
            \omega^1 := dy - p \,dx \\
            \omega^2 := dz - F(x, y, p, q, z) \,dx - F_q(x, y, p, q, z) (dp - q \, dx) \\
            \omega^3 := dp - q \,dx \\
            \omega^4 := dq \\
            \omega^5 := dx \textrm{.}
        \end{array}
    \right.
\end{equation}
The frame $(E_a)$ of $TM$ dual to $(\omega^a)$ is
\begin{equation}\label{equation:local-frame}
    \left\{
        \begin{array}{l}
            E_1 := \partial_y \\
            E_2 := \partial_z \\
            E_3 := \partial_p + F_q(x, y, p, q, z) \partial_z \\
            E_4 := \partial_q \\
            E_5 := \partial_x + p \partial_y + q \partial_p + F(x, y, p, q, z) \partial_z \textrm{.}
        \end{array}
    \right.
\end{equation}
In particular,
\[
    D_F = \ker\set{\omega^1, \omega^2, \omega^3} = \ab{E_4, E_5} = \ab{\partial_q, \partial_x + p \partial_y + q \partial_p + F(x, y, p, q, z) \partial_z} \textrm{.}
\]

\section{Some metric and conformal geometry}\label{section:metric-conformal-geometry}

Conformal geometry is the geometry of smooth manifolds in which one has a notion of angle but (in particular) not of length.

\begin{definition}
A \textbf{conformal structure} on a smooth manifold $M$ is an equivalence class $c$ of (pseudo-Riemannian) metrics under the relation $\sim$, where $g \sim \whg$ if and only if $\whg = \Omega^2 g$ for some positive function $\Omega \in C^{\infty}(M)$, and the pair $(M, c)$ is called a \textbf{conformal manifold}. Any metric $g \in c$ is a \textbf{(conformal) representative} of $c$. The \textbf{signature} $(p, q)$ of a conformal structure $c$ is just the signature of any (equivalently, every) conformal representative.
\end{definition}

An \textbf{infinitesimal symmetry} of a conformal structure $c$ on an $n$-manifold $M$, (alternatively, a \textbf{conformal Killing field} on $(M, c)$) is a vector field $\xi \in \Gamma(TM)$ that preserves the conformal structure in the sense that for a representative metric $g \in c$, $\mcL_{\xi} g = \lambda g$ for some $\lambda \in C^{\infty}(M)$. Taking traces gives that this condition is equivalent to $\tf(\mcL_{\xi} g) = 0$, where $\tf(S)$ denotes the tracefree part $S_{ab} - \frac{1}{n} S_c^{\phantom{c} c} g_{ab}$ of $S$, and direct computation shows that this condition is independent of the choice of representative $g$. We denote the space of infinitesimal symmetries of $c$ by $\mfaut(c)$; checking shows that it is closed under the Lie bracket of vector fields, so it is a Lie subalgebra under that operation.

The \textbf{metric bundle} of the conformal structure $c$ on a smooth manifold $M$ is the ray bundle $\pi: \mcG \to M$ defined by
\[
    \mcG := \coprod_{x \in M} \set{g_x : g \in c} \textrm{.}
\]
By construction, the sections of $\mcG$ are precisely the representative metrics $g$ of $c$. The action $\bbR_+ \times \mcG \to \mcG$ defined by $s \cdot g_x = \delta_s(g_x) := s^2 g_x$ naturally realizes $\mcG$ as a principal $\bbR_+$-bundle; we denote the infinitesimal generator of this dilation by $\bbT := \partial_s \delta_s \vert_{s = 1}$. The metric bundle admits a tautological degenerate, symmetric $2$-tensor $\smash{\mbg_0 \in \Gamma(\odot^2 T^* \mcG)}$ defined by $(\mbg_0)_{g_x} (\xi, \eta) := g_x(T_{g_x} \pi \cdot \xi, T_{g_x} \pi \cdot \eta)$, which we may identify with the conformal structure itself.

Fixing a representative $g \in c$ yields a trivialization $\mcG \cong M \times \bbR_+$ by identifying the inner product $t^2 g_x$ with $(x, t)$. In this trivialization, the tautological $2$-tensor is $\mbg_0 = t^2 \pi^* g$, the dilations are given by $\delta_s: (t, x) \mapsto (s t, x)$, and the infinitesimal generator is $\bbT = t \partial_t$.

For any $w \in \bbR$ the \textbf{conformal density bundle of weight $w$} is the bundle $\mcD[w] := \mcG \times_{\rho_{-w}} \bbR$ associated to $\mcG$ by the $\bbR_+$-representation $\rho_{-w}(y) := s^{-w} y$. We may identify this bundle with $\coprod_{x \in M} \set{f : \mcG_x \to \bbR : \delta_s^* f = s^w f, \, s \in \bbR_+}$ and hence its sections with real-valued functions on $\mcG$ of homogeneity $w$ (with respect to the dilations $\delta_s$). A choice of representative $g \in c$ induces a trivialization of each density bundle $\mcD[w]$ by identifying $f \in \Gamma(\mcD[w])$ with $f \circ g \in C^{\infty}(M)$, where we regard $g$ as a section $M \to \mcG$.

Given a vector bundle $E \to M$ and any $w \in \bbR$, we may form a conformally weighted vector bundle $E[w] := E \otimes \mcD[w]$, and again a choice of representative trivializes any such bundle by identifying $v \otimes f \in \Gamma(E[w])$ with $(f \circ g) v \in \Gamma(E)$. By construction, $\mbg_0$ satisfies $\delta_s^* \mbg_0 = s^2 \mbg_0$ and depends only on the $T\pi$-fibers of its arguments, and unwinding definitions shows that we may view the conformal structure itself as a canonical section $\mbg$ of $\smash{\odot^2 T^* M [2]}$.

The class of conformal structures considered in this work all admit an additional special structure called an almost Einstein scale.

A metric $g$ on an $n$-manifold $M$ is \textbf{Einstein} if $\Ric = 2 \lambda (n - 1) g$ for some smooth function $\lambda \in C^{\infty}(M)$. If $n \geq 3$ and $M$ is connected, then $\lambda$ is necessarily constant; this is the \textbf{Einstein constant} of $g$, though this term is usually used for the full coefficient $2 \lambda (n - 1)$.

\begin{definition}
Suppose $(M, c)$ is a conformal manifold. An \textbf{Einstein scale} for $c$ is a (nonvanishing) weighted smooth function $\sigma \in \Gamma(\mcD[1])$ such that the (unweighted) metric $\sigma^{-2} \mbg$ is Einstein, and if $c$ admits an Einstein scale, we say that $c$ is \textbf{(conformally) Einstein}. A weighted smooth function $\sigma \in \Gamma(\mcD[1])$ with zero set $\Sigma$ is an \textbf{almost Einstein scale} for $c$ if $\sigma^{-2} \mbg\vert_{M - \Sigma}$ is Einstein. We denote the space of almost Einstein scales of $c$ by $\mbaEs(c)$. If $\dim M \geq 3$ and $M$ is connected, and if the almost Einstein scale $\sigma$ is not identically zero, every restriction of $\sigma^{-2} \mbg\vert_{M - \Sigma}$ to a connected component of $M - \Sigma$ has the same Einstein constant, which we hence call the \textbf{Einstein constant} of $\sigma$; if $\lambda = 0$, we say that $\sigma$ is an \textbf{almost Ricci-flat scale}. (For expository convenience, we declare the identically zero almost Einstein scale to be almost Ricci-flat too.)
\end{definition}

\subsection{Conformal tractor and ambient geometry}\label{subsection:conformal-tractor-ambient-geometry}

The conformal tractor bundle is a construction that encodes a conformal manifold $(M, c)$ of signature $(p, q)$ in a rank-$(p + q + 2)$ bundle $\mcT \to M$ endowed with some auxiliary structure. The closely related ambient metric construction assigns to $(M, c)$ a pseudo-Riemannian manifold $(\wtM, \wtg)$ of signature $(p + 1, q + 1)$; this latter construction involves some choices, but the nonuniqueness in its construction is manageable. Invariant data extracted from either construction (and which, in the case of the ambient metric, do not depend on any choices made) are invariants of the underlying conformal structure and hence can be used to analyze that structure; indeed, this was the original motivation for the construction of the ambient metric \cite{FeffermanGraham1985}.

We first describe the ambient metric associated to $c$ (for $n$ odd, which is all that we need here), following the standard reference \cite{FeffermanGraham2011}, and then use it to construct the standard conformal tractor bundle as described in \cite{CapGover2003}. The conformal tractor construction was first given \cite{Thomas1926} and was rediscovered and extended in \cite{BaileyEastwoodGover1994} using an approach different from the one here, and it can also be described in the language of parabolic geometries \cite{CapSlovak2009}. The discussion in \cite[\S 2]{GrahamWillse2012G2} is similar to the one below but includes the case of even $n$.

Henceforth in this subsection, $(M, c)$ is a conformal structure of dimension $n \geq 3$ and signature $(p, q)$. Consider the space $\mcG \times \bbR$, and denote the standard coordinate on $\bbR$ by $\rho$. The dilations $\delta_s$ extend to $\mcG \times \bbR$ by acting on the $\mcG$ factor, that is, by $\delta_s(g_x, \rho) = s \cdot (g_x, \rho) := (s^2 g_x, \rho)$, and we again denote its infinitesimal generator $\bbT := \partial_s \delta_s \vert_{s = 1}$. The map $\mcG \hookrightarrow \mcG \times \bbR$ defined by $z \mapsto (z, 0)$ embeds $\mcG$ as a hypersurface in $\mcG \times \bbR$, and we identify $\mcG$ with its image, $\mcG \times \set{0}$, under this map. Again a choice of representative $g$ induces a trivialization $\mcG \times \bbR \leftrightarrow \bbR_+ \times M \times \bbR$ by identifying $(t^2 g_x, \rho) \leftrightarrow (t, x, \rho)$, which defines an embedding $M \hookrightarrow \mcG \times \bbR$ by $x \mapsto (1, x, 0)$ and yields an identification $T(\mcG \times \bbR) \cong \bbR \oplus TM \oplus \bbR$.

A smooth metric $\wtg$ of signature $(p + 1, q + 1)$ on an open neighborhood $\wtM$ of $\mcG$ in $\mcG \times \bbR$ invariant under the dilations $\delta_s$, $s \in \bbR_+$, is a \textbf{pre-ambient metric} for $(M, c)$ if (1) it extends $\mbg_0$, that is, if $\iota^* \wtg = \mbg_0$, and (2) if it has the same homogeneity as $\mbg_0$ with respect to the dilations, that is, if $\delta_s^* \wtg = s^2 \wtg$ (again, for all $s$). A pre-ambient metric is \textbf{straight} if for all $p \in \wtM$ the parametrized curve $s \mapsto s \cdot p$ is a geodesic. Any nonempty conformal structure admits many pre-ambient metrics; Cartan's normalization condition for a conformal connection \cite{Cartan1923} suggests that Ricci-flatness is a natural distinguishing criterion.

\begin{definition}
Let $(M, c)$ be a conformal manifold of odd dimension at least $3$. An \textbf{ambient metric} for $(M, c)$ is a straight pre-ambient metric $\wtg$ such that $\Ric(\wtg)$ is $O(\rho^{\infty})$; the pair $(\wtM, \wtg)$ is an \textbf{ambient manifold} for $(M, c)$.
\end{definition}

Here, we say that a tensor field on $\wtM$ is $O(\rho^\infty)$ if it vanishes to infinite order at each point of the zero set of $\rho$.  We formulate Fefferman-Graham's existence and uniqueness results for ambient metrics of odd-dimensional conformal structures as follows:

\begin{theorem}\cite{FeffermanGraham2011}
Let $(M, c)$ be a conformal manifold of odd dimension at least $3$. There exists an ambient metric for $(M, c)$, and it is unique up to pullback by diffeomorphisms that restrict to $\id_{\mcG}$ and up to infinite order: If $\wtg_1$ and $\wtg_2$ are ambient metrics for $(M, c)$, then (after possibly restricting the domains of both to appropriate open neighborhoods of $\mcG$ in $\wtM$) there is a diffeomorphism $\phi$ such that $\phi\vert_{\mcG} = \id_{\mcG}$ and $\phi^* \wtg_2 - \wtg_1$ is $O(\rho^{\infty})$.
\end{theorem}


We now recover the standard tractor bundle from the ambient construction. Let $\wtg$ be an ambient metric for $(M, c)$. Since $\wtg$ is straight, the fiber $\mcG_x$ of $\mcG \to M$ is a geodesic of $\wtg$ (with geodesic parametrization $s \mapsto s \cdot g_x$ for any $g_x \in \mcG_x$). Direct computation shows that a vector field $\xi$ along $\mcG_x$ in $\wtM$ is parallel if and only if $T\delta_s \cdot \xi = s \xi$ for all $s \in \bbR_+$ (see \cite[\S2]{GrahamWillse2012G2} for details). We define the standard tractor bundle to be the bundle whose smooth sections are vector fields along $\mcG$ in $\wtM$ with this homogeneity with respect to $\delta_s$.

\begin{definition}
Let $(M, c)$ be a conformal manifold. The \textbf{(standard conformal) tractor bundle} is the bundle $\pi: \mcT \to M$ defined by
\[
    \mcT = \coprod_{x \in M} \set{\xi \in \Gamma(T\wtM\vert_{\mcG_x}) : T\delta_s \cdot \xi = s \xi, s \in \bbR_+} \textrm{.}
\]
We call a section of $\mcT$ a \textbf{(standard) (conformal) tractor (field)}.
\end{definition}

We construct some additional natural objects on $\mcT$. The sections of $\mcT$ are the vector fields on $T\mcG$ that satisfy $\delta_s^* X = s^{-1} X$; since by definition $\delta_s^* \wtg = s^2 \wtg$, if $X, Y \in \mcT_x$, then $\wtg(X, Y)$ is constant. So, the restriction of $\wtg$ to $\mcG$ defines a fiber metric $g^{\mcT}$ (of signature $(p + 1, q + 1)$) on $\mcT$. Now, $\bbT$ has homogeneity $0$, so we may regard it as a section of $\mcT[1]$, and the span of $\bbT \in \Gamma(T \mcG)$ is invariant under $\delta_s$, so it descends to a distinguished line subbundle of $\mcT$, which by mild abuse of notation we call $[\bbT]$.

The Levi-Civita connection $\wtnabla$ of the ambient metric defines the \textbf{tractor connection} $\nabla^{\mcT}$ on $\mcT$ as follows. First note that the map $T\pi: T\mcG \to TM$ induces a realization of the tangent bundle by
\[
    T_x M = \set{\eta \in \Gamma(T\mcG\vert_{\mcG_x}) : T\delta_s \cdot \eta = \eta, \, s \in \bbR_+} / [\bbT\vert_{\mcG_x}] \textrm{.}
\]
Now, for $X \in \Gamma(\mcT)$ and $\xi \in \Gamma(TM)$, define
\[
    \nabla^{\mcT}_{\xi} X := \wtnabla_{\olxi} X \textrm{,}
\]
where $\olxi$ is an arbitrary lift of $\xi$ to $\set{\eta \in \Gamma(T\mcG\vert_{\mcG_x}) : T\delta_s \cdot \xi = \xi, \, s \in \bbR_+}$ in the above realization of $TM$. Direct computation verifies that the right-hand side is independent of the choice of lift, that it is a section of $\mcT$, and that $\nabla^{\mcT}$ is a vector bundle connection. This construction of the conformal tractor bundle depends on the choice of ambient metric $\wtg$, but different choices yield equivalent constructions. A proof that this construction is equivalent to the standard tractor bundle of \cite{BaileyEastwoodGover1994} is given in \cite{CapGover2003}. It is a direct consequence of the compatibility of the ambient metric with its Levi-Civita connection that the tractor metric and tractor connection are likewise compatible in that they satisfy $\nabla^{\mcT} \mbg^{\mcT} = 0$.

Given a tractor $\chi \in \Gamma(\mcT)$, counting homogeneities shows that we may regard $\mbg(\chi, \bbT)$ as a section of $\mcD[1]$, so this defines a tensorial \textbf{canonical projection} $\Pi_0: \Gamma(\mcT) \to \Gamma(\mcD[1])$; By construction, the kernel of this map is the space of sections of $\smash{\Gamma([\bbT]^{\perp})}$, so $\Pi_0$ descends to a natural bundle isomorphism $\smash{\mcT / [\bbT]^{\perp} \cong \mcD[1]}$. Conversely, there is a natural map $L_0: \Gamma(\mcD[1]) \to \Gamma(\mcT)$, called the \textbf{BGG splitting operator}, which is not tensorial but depends on the $2$-jet of a section of $\mcD[1]$, that satisfies $L_0(\Pi_0(\chi)) = \chi$ for any parallel tractor $X \in \Gamma(\mcT)$ (see, for example, \cite[\S\S2.5, 3.2]{HammerlSagerschnig2009}). In particular, any parallel tractor is determined by its image under the canonical projection. Furthermore, this image admits a natural geometric interpretation.

\begin{proposition}
The restrictions of the maps $\Pi_0$ and $L_0$ are isomorphisms
\begin{equation}\label{equation:bijection-parallel-tractors-almost-Einstein-scales}
    \set{\textrm{parallel tractors}} \mathop{\rightleftarrows}^{\Pi_0}_{L_0} \set{\textrm{almost Einstein scales}}
\end{equation}
of vector spaces.
\end{proposition}

In fact, the first modern formulation \cite{BaileyEastwoodGover1994} of the tractor bundle constructs $\mcT$ and $\nabla^{\mcT}$ so that almost Einstein scales are exactly the (weighted) smooth functions whose suitable prolongations yield sections of $\mcT$ parallel with respect to $\nabla^{\mcT}$.

If $\sigma$ is a nonzero almost Einstein scale of Einstein constant $\lambda$, the parallel tractor $L_0(\sigma)$ satisfies $\ab{L_0(\sigma), L_0(\sigma)} = -2 \lambda$ (this is immediate using a trivialization of the tractor bundle that we do not describe here; see, for example, \cite{Leistner2006}). In particular, $L_0(\sigma)$ is null if and only if $\sigma$ is an almost Ricci-flat scale.

\subsection{Holonomy}\label{subsection:holonomy}

We briefly review and relate several notions of holonomy; see \cite[Chapter 2]{KobayashiNomizu1996} for a more detailed discussion.

Let $E \to N$ be a vector bundle of rank $k$ over a field $\bbF$, and let $\nabla^E : \Gamma(E) \to \Gamma(E \ot T^* N)$ be a general connection on $E$. For any piecewise smooth curve $\gamma: [0, 1] \to N$, a local section $\sigma$ of $E$ is parallel along $\gamma$ if $\nabla^E_{\gamma'} \sigma = 0$. This is a first-order linear ordinary differential equation with piecewise smooth coefficients, so we define the linear \textbf{parallel transport} map $P_{\gamma}: E_{\gamma(0)} \to E_{\gamma(1)}$ that sends $\xi \in E_{\gamma(0)}$ to $\sigma(\gamma(1))$, where $\sigma$ is the unique solution to the differential equation that satisfies $\sigma(\gamma(0)) = \xi$. If $\gamma$ is a loop based at $u \in N$, that is, if $\gamma(0) = \gamma(1) = u$, then since $P_{\gamma}$ is invertible (its inverse is just $P_{\gamma^{-1}}$, where $\gamma^{-1}(t) := \gamma(1 - t)$), $P_{\gamma^{-1}} \in \GL(E_u)$. Let $\Omega_u$ denote the space of loops based at $u$.

\begin{definition}\label{definition:holonomy}
Let $E \to N$ be a vector bundle with general connection $\nabla^E$, and let $u$ be a point in $N$. The \textbf{holonomy} of $\nabla^E$ based at $x$ is the group
\[
    \Hol_u(\nabla^E) := \set{P_{\gamma} : \gamma \in \Omega_u(N)} \leq \GL(E_u) \textrm{.}
\]
\end{definition}

Picking a basis of $E_u$ identifies it with $\bbF^k$ and so realizes $\Hol_u(\nabla^E)$ as an explicit subgroup of $\GL(n, \bbF)$; any other choice of basis yields a conjugate subgroup, so without reference to bases we may regard $\Hol_u(\nabla^E)$ as a conjugacy class of subgroups of $\GL(n, \bbF)$. If $N$ is connected and $v$ is another point in $N$, let $\alpha$ be any path from $u$ to $v$; then by construction
\begin{equation}\label{equation:holonomy-conjugacy}
    \Hol_v(\nabla^E) = P_{\alpha} \Hol_u(\nabla^E) P_{{\alpha}^{-1}} \textrm{,}
\end{equation}
so that conjugacy class is also independent of the base point $u$. So, for connected manifolds, we may suppress reference to a base point and call that conjugacy class the holonomy $\Hol(\nabla^E)$ of $\nabla^E$. For each $x$, we may regard the fiber $E_u$ as a representation of $\Hol_u(\nabla^E)$, whence \eqref{equation:holonomy-conjugacy} shows that the isomorphism type of the representation does not depend on the choice of base point.

If $U \subset N$ is an open subset containing $u$, then all the loops in $U$ based at $u$ are loops in $N$ based there, and thus $\Hol_u(\nabla^E\vert_U) \leq \Hol_u(\nabla^E)$. So, we define the \textbf{local holonomy} of $\nabla^E$ based at $u$ to be the group
\[
    \Hol^*_u(\nabla^E) := \bigcap_{\beta} \Hol_u(\nabla^E\vert_{U_{\beta}})
\]
where $\beta$ indexes the set of open subsets $U_{\beta} \subseteq N$ that contain $u$. The local holonomy is a connected subgroup of $\GL_+(T_u E)$, and there is some open set $V$ containing $u$ such that $\Hol_u^*(\nabla^E) = \Hol_u(\nabla^E\vert_V)$; we denote its Lie algebra by $\mfhol^*_u(\nabla^E)$.

If a section $\sigma$ is a parallel, then $\sigma_u$ is fixed by $\Hol_u(\nabla^E)$; conversely, any $s \in E_u$ fixed by $\Hol_u(\nabla^E)$ defines by parallel transport a parallel section $\sigma$ such that $\sigma_u = s$. Thus, for a parallel section $\sigma$, $\Hol_u(\nabla^E) \leq \Stab_{\GL(E_u)} (\sigma_u)$. Analogous statements hold for sections of tensor bundles $\otimes^k E^*$ parallel with respect to the connections induced by $\nabla^E$.

The holonomy of a vector bundle connection $\nabla^E$ is closely related to its curvature $R^E$.

\begin{theorem}[Ambrose-Singer]\label{theorem:ambrose-singer}
Let $\nabla^E$ be a connection on a vector bundle $E \to M$, denote by $R^E$ its curvature viewed as a section of $\otimes^2 T^*M \to \End(E)$, and fix $u \in M$. The underlying vector space of the Lie algebra
\[
    \mfhol_u(\nabla^E) \leq \End(E_u)
\]
of the holonomy group $\Hol_u(\nabla^E)$ is spanned by the elements
\[
    P_{\gamma}^{-1} R^E_{\gamma(1)} (X, Y) P_{\gamma} \textrm{,}
\]
such that $\gamma: [0, 1] \to M$ is a smooth curve such that $\gamma(0) = u$, $P_{\gamma}$ is the parallel transport map $E_u \to E_{\gamma(1)}$ defined by the connection that $\nabla^E$ induces on $\End(E)$, and $X, Y \in T_{\gamma(1)} M$.
\end{theorem}

See, for example, \cite[Chapter 2, Theorem 8.1]{KobayashiNomizu1996} for the principal bundle version of this theorem.

One can recover partial information about the holonomy of $\nabla^E$ at $x$ to infinite order by differentiating curvature the curvature $R^E$, and it will turn out that this weaker version of holonomy will be all we need to prove the main result below.

\begin{definition}
Let $\nabla^E$ be a vector bundle connection on a vector bundle $E \to M$, and fix $u \in M$. The \textbf{infinitesimal holonomy (algebra)} of $\nabla^E$ is the Lie subalgebra $\mfhol'_u(\nabla^E) \leq \mfgl(E_u)$ generated by the endomorphisms
\[
    (R^E_u)_{ab \phantom{C} D, m_1 \cdots m_k}^{\phantom{ab} C} X^a Y^b Z_1^{m_1} \cdots Z_k^{m_k}
\]
for all $k \geq 0$ and $X, Y, Z_1, \ldots, Z_k \in T_u N$, where $C$ and $D$ are indices on the fiber $E_u$, and where the covariant derivatives are taken with appropriate connections on $\Lambda^2 T^* M \otimes E \otimes (\otimes^k T^* M)$ given by coupling the connection on $\End(E)$ induced by $\nabla^E$ with an arbitrary connection on $M$.
\end{definition}

By construction, $\mfhol'_u(\nabla^E) \leq \mfhol^*_u(\nabla^E) \leq \mfhol_u(\nabla^E)$ (see \cite[Chapter 2, Proposition 10.4]{KobayashiNomizu1996} for a proof of the first containment for principal connections). We then call the connected subgroup $\Hol'_u(\nabla^E) \leq \GL(E_u)$ with Lie algebra $\mfhol'_u(\nabla^E)$ the \textbf{infinitesimal holonomy (group)} of $\nabla^E$, and by construction, $\Hol'_u(\nabla^E) \leq \Hol_u(\nabla^E)$. (In fact, this becomes an equality when $N$ is simply connected and the underlying data is real-analytic.)

We will invoke the holonomy construction for two particular connections. Given a pseudo-Riemannian metric $h$ of signature $(p, q)$ on an $n$-manifold $N$, the holonomy of $h$, which we denote $\Hol_u(h)$ (or just $\Hol(h)$ if we only care about its conjugacy class) is just the holonomy $\Hol_u(\nabla^h)$ of its Levi-Civita connection, $\nabla^h$. Since $h$ itself is parallel, the induced action of $\Hol_u(h)$ must preserve $h_u$, that is, $\Hol_u(h) \leq \Ooperator(h_u)$; passing to conjugacy classes gives $\Hol(h) \leq \Ooperator(p, q)$.

A $k$-plane field $S$ is said to be parallel with respect to a connection $\nabla$ if for all sections $\xi \in \Gamma(S)$ and all vectors $\eta \in TN$ we have $\nabla_{\eta} \xi \in S$. By construction, if a $k$-plane field $S$ is parallel, then $\Hol_u(\nabla)$ fixes $S_u$; conversely, if $\Hol_u(\nabla)$ preserves a $k$-plane $s \subset T_u N$, by parallel transport it defines a parallel plane field $S$ such that $S_u = s$. In particular, if $S$ is parallel, then $S_u \leq T_u N$ is a subrepresentation of $\Hol_u(\nabla)$; as usual if $S$ is proper we say that $\Hol_u(h)$ acts reducibly. One can show that any parallel plane field is integrable, so $S$ defines a foliation of $N$ by $k$-manifolds. Since the holonomy preserves $h_u$ and $S_u$, the $(n - k)$-plane field $S^{\perp}$ is parallel too. Then, if $S$ is nondegenerate---that is, if $S \cap S^{\perp} = \set{0}$, which in particular is always the case if $h$ is definite---the $\Hol_u(h)$-representation $T_u N$ decomposes as $T_u N = S_u \oplus S^{\perp}_u$. In this case, one can show that there is an open set $U$ containing $u$ and integral manifolds $V$ and $W$ of $S$ and $S^{\perp}$ through $u$, respectively, such that $(U, h\vert_U)$ is locally isometric to $(V \times W, h_V \oplus h_W)$, where $h_V$ and $h_W$ are the respective pullbacks of $h$ to those leaves. Then, by construction, $\Hol(h\vert_U) = \Hol(h_V) \times \Hol(h_W)$, so to understand pseudo-Riemannian metric holonomy (locally, anyway) it is enough to understand indecomposably acting holonomy groups, that is, those for which the representation $T_u N$ of $\Hol_u(\nabla)$ does not decompose as a direct sum of proper subrepresentations, where $\nabla$ here denotes the Levi-Civita connection. There is no full classification of these groups, but there is a complete list of \textit{irreducibly} acting holonomy groups of metrics, at least on simply connected manifolds, and therefore of irreducibly acting metric local holonomy groups. (There are significant partial results, however, about which groups can occur in the remaining case, that is, the local metric holonomy groups that act indecomposably but not irreducibly.)

\begin{theorem}[Berger's List]\cite{Berger1955}
Let $(N, h)$ be a simply connected pseudo-Riemannian $n$-manifold of signature $(p, q)$ that is not locally a symmetric space (a pseudo-Riemannian manifold with parallel curvature tensor). If $\Hol(N, h)$ acts irreducibly, then up to isomorphism it is one of the following: $\SO(p, q)$, $\U(\frac{p}{2}, \frac{q}{2})$, $\SU(\frac{p}{2}, \frac{q}{2})$, $\Sp(\frac{p}{4}, \frac{q}{4})$, $\Sp(\frac{p}{4}, \frac{q}{4}) \cdot \Sp(1)$, $\SO(\frac{n}{2}, \bbC)$, $\G_2^c$ (only in signatures $(7, 0)$ and $(0, 7)$), $\G_2$ (signatures $(3, 4)$ and $(4, 3)$), $\G_2^{\bbC}$ (only in real dimension 14), $\Spin(7)$ (signatures $(8, 0)$ and $(0, 8)$), $\Spin(3, 4)$ (signature $(4, 4)$), and $\Spin(7, \bbC)$ (real dimension $16$).
\end{theorem}

(In fact, Berger's List originally included several other groups, but these were all later shown to occur only for symmetric spaces \cite{Bryant1987}.) By the above discussion, if the holonomy $\Hol(N, h)$ of a pseudo-Riemannian manifold acts indecomposably but not irreducibly, it must admit some proper degenerate $k$-plane field $S$, and thus a proper parallel totally null plane field, namely, $S \cap S^{\perp}$. In particular, locally all holonomy groups of metrics that are not locally symmetric and that act indecomposably but do not appear on Berger's list must admit a proper parallel totally null plane field.

Now, given a conformal structure $c$ of signature, say, $(p, q)$, the \textbf{conformal holonomy} or \textbf{tractor holonomy} of $c$ is the holonomy $\Hol(\nabla^{\mcT})$ of its tractor connection. Since $\nabla^{\mcT}$ is compatible with the (signature-$(p + 1, q + 1)$) tractor metric $g^{\mcT}$, $\Hol(\nabla^{\mcT}) \leq \Ooperator(p + 1, q + 1)$.

Given furthermore a choice of ambient metric $\wtg$ for $c$, we call $\Hol(\wtg)$ the \textbf{ambient holonomy} of $c$, though this choice in general depends on the choice of $\wtg$. It follows from the construction of the tractor connection from the ambient metric that $\Hol(\nabla^{\mcT}) \leq \Hol(\wtg)$ for all ambient metrics $\wtg$ of $c$ \cite{CapGover2003}.

We will use the following recent result of \v{C}ap, Gover, Graham, and Hammerl.

\begin{theorem}\cite{CapGoverGrahamHammerl2011}\label{theorem:CGGH}
Let $(M, c)$ be an conformal manifold of odd dimension at least $3$, fix $u \in M$, let $\wtg$ be an ambient metric for $c$, and let $z \in \mcG$ be an element in the fiber of $\mcG \to M$ over $u$. Then, $\mfhol'_u(\nabla^{\mcT}) \cong \mfhol'_z(\wtg)$.
\end{theorem}

There is an analogue of this theorem for even-dimensional manifolds, but the statement is more subtle and we do not need it here.

\section{Nurowski conformal structures}\label{section:Nurowski-conformal-structures}

In \cite{Nurowski2005}, Nurowski showed that to any generic $2$-plane field $D$ on a $5$-manifold $M$ one can associate a canonical conformal structure $c_D$ of signature $(2, 3)$ on $M$; we call any conformal structure that arises in this way a \textbf{Nurowski conformal structure}. One can read from Nurowski's original formulation that with respect to $c_D$, $D$ is totally null and $\smash{[D, D] = D^{\perp}}$.

Given a $2$-plane field $D_F$ defined by a smooth function $F$ via the quasi-normal form, \eqref{equation:local-coframe-partial}, we denote the conformal structure it induces via Nurowski's construction by $c_F := c_{D_F}$. Nurowski's remarkable formula for a representative $g_F$ of $c_F$ in the coframe $(\omega^a)$ is a sextic polynomial in the components of the $4$-jet of $F$; it contains more than 70 terms, however, so we do not reproduce it here.

One can realize Nurowski's conformal construction much more compactly at the cost of (substantial) abstraction, using the language of parabolic geometries. Cartan showed that a conformal structure of signature $(p, q)$ on a manifold $M$ can be realized as a principal $\dot{P}$-bundle over $M$ endowed with a $\mfo(p + 1, q + 1)$-valued Cartan connection satisfying a natural normalization condition (this connection is called the normal conformal connection), where $\dot{P}$ is the stabilizer of a null line in $\bbR^{p + q + 2}$ endowed with an inner product of signature $(p + 1, q + 1)$. The motivating feature of parabolic geometry (and more generally, Cartan geometry) is that it realizes many geometries in a common framework, allowing them to be treated in a unified way.

A generic $2$-plane field $D$ on a $5$-manifold $M$ can be realized as a principal $P_1$-bundle $E \to M$ endowed with a $\mfg_2$-valued Cartan connection $\omega$ satisfying some normalization criteria (recall that $P_1$ is the stabilizer in $\G_2$ of a null line in $\ImwtbbO$). Recall too that there is a natural embedding $\G_2 \hookrightarrow \SO(3, 4)$. By construction, $P_1 = \dot{P} \cap \G_2$, so we may extend any such $E$ by forming the principal $\dot{P}$-bundle $\dot{E} := E \times_{P_1} \dot{P}$; then, we can extend $\omega$ to a $\mfo(3, 4)$-valued $1$-form $\dot{\omega}$ on $\dot{E}$ by equivariance, and that equivariance guarantees that $\dot{\omega}$ is a Cartan connection. One can furthermore show that $\dot{\omega}$ satisfies Cartan's normalization condition, so it corresponds to a conformal structure $c_D$ of signature $(2, 3)$.

The parabolic perspective also yields an explicit method for recovering for a $2$-plane field $D$ the Cartan curvature tensor $A \in \Gamma(\odot^4 D^*)$ described in Subsection \ref{subsection:Cartan-curvature-tensor}: The conformal structure $c_D$ can be encoded in a $\CO(2, 3)$-principal frame bundle over $M$, and its Weyl curvature $W$ (which is the harmonic curvature of the parabolic geometry corresponding to conformal geometry in dimension $n \geq 4$) can be regarded as a section of the bundle associated to the representation of Weyl-type symmetries tensored with an appropriate conformal density bundle. The underlying $2$-plane field $D$ determines a reduction of the frame bundle to $\GL(2, \bbR)$ and hence induces natural decompositions of the corresponding $\CO(2, 3)$-representations into $\GL(2, \bbR)$-representations. In particular, the Weyl curvature decomposes into 15 pieces; all but five pieces turn out always to be zero, and the highest-weight summand is exactly Cartan's $A$ up to a nonzero constant factor.

One can use this Cartan connection approach to show that the conformal structure $c_{\Delta}$ that the flat model $(\mcQ, \Delta)$ induces on $\mcQ$ is just the conformally flat structure induced by the conformally flat structure on $\bbS^2 \times \bbS^3$ via the projection $\bbS^2 \times \bbS^3 \to \mcQ$.

\subsection{Characterization of Nurowski conformal structures}

Nurowski conformal structures $(M, c)$ are characterized in terms of objects on the base manifold \cite{HammerlSagerschnig2009}; we give a version (that is essentially given in that reference) of that characterization in terms of tractor data. A tractor $3$-form is a section of the bundle $\Lambda^3 \mcT^*$, and the tractor connection induces a connection, which we also denote $\nabla^{\mcT}$, on that bundle. Given any vector bundle $S \to M$ and any section $\varphi \in \Gamma(\Lambda^3 S^*)$ such that $\varphi_x$ is generic for all $u \in M$, let $H(\varphi) \in \Gamma(\odot^2 S^*)$ denote the bilinear form defined respectively on each fiber $S_u$ by $H(\varphi)_u := H(\varphi_u)$, where $H$ is the map \eqref{equation:induced-bilinear-form}.

\begin{theorem}\cite[Theorem A]{HammerlSagerschnig2009}\label{theorem:HS-Theorem-A}
A conformal structure $c$ on an oriented $5$-manifold $M$ is Nurowski (that is, there is a generic $2$-plane field $D$ on $M$ such that $c = c_D$) if and only if there is a parallel tractor $3$-form $\Phi \in \Gamma(\Lambda^3 \mcT^*)$ compatible with the tractor metric $g^{\mcT}$ in the sense that $g^{\mcT} = H(\Phi)$.
\end{theorem}

For any conformal manifold $(M, c)$, just as for the (standard) tractor bundle $\mcT$, there is a canonical projection $\Pi_0: \Gamma(\Lambda^3 \mcT^*) \to \Gamma(\Lambda^2 T^* M[3])$ and an operator $L_0: \Gamma(\Lambda^2 T^* M[3]) \to \Gamma(\Lambda^3 \mcT^*)$ so that $L_0(\Pi_0(\chi)) = \chi$ for every parallel tractor $3$-form $\chi \in \Gamma(\Lambda^3 \mcT^*)$. In particular such a $3$-form is determined by its image under $\Pi_0$, and the weighted $2$-forms produced this way are exactly the so-called normal conformal Killing $2$-forms of $c$ (see \cite{Leitner2005, HammerlSagerschnig2009}). Given a $2$-plane field $D$ on an oriented $5$-manifold $M$ and a corresponding parallel tractor $3$-form as in Theorem \ref{theorem:HS-Theorem-A}, the compatibility condition forces the $2$-form $\Pi_0(\Phi)$ to be nonvanishing and locally decomposable, so its kernel is a $3$-plane field on $M$, and this turns out to be $[D, D]$. Since $\smash{[D, D]^{\perp} = D}$ (where $\perp$ denotes the orthogonal with respect to the induced conformal structure $c_D$), $D$ is exactly the $2$-plane field spanned by the (locally decomposable, weighted) bivector field given by raising both indices of $\Pi_0(\Phi)$.

By naturality, every infinitesimal symmetry of a generic $2$-plane field $D$ on a $5$-manifold is also an infinitesimal symmetry of the induced conformal structure $c_D$, defining a natural inclusion $\mfaut(D) \hookrightarrow \mfaut(c_D)$. Hammerl and Sagerschnig showed one can realize the space of almost Einstein scales of $c_D$ as a natural complement to $\mfaut(D)$ in $\mfaut(c_D)$:

\begin{theorem}\cite[Theorem B]{HammerlSagerschnig2009}\label{theorem:HS-Theorem-B}
Let $D$ be a generic $2$-plane field on a $5$-manifold $M$, let $c_D$ denote the Nurowski conformal structure it induces, let $\Phi$ be a parallel $3$-form characterizing $D$ as in Theorem \ref{theorem:HS-Theorem-A}, and denote $\phi := \Pi_0(\Phi)$. There is a natural map $\mbaEs(c_D) \hookrightarrow \mfaut(c_D)$ defined by
\begin{equation}\label{equation:aEs-to-cKf}
    \sigma \mapsto \phi^{ab} \sigma_b + \tfrac{1}{4} \phi^{ba}_{\phantom{ba}, b} \sigma
\end{equation}
where covariant derivatives are taken with respect to the Levi-Civita connection of an arbitrary representative $g \in c_D$, and where $\sigma_b := \nabla_b \sigma$. If we identify $\mbaEs(c_D)$ with its image under this map, it is complementary to $\mfaut(D)$ in $\mfaut(c_D)$:
\[
    \mfaut(c_D) = \mfaut(D) \oplus \mbaEs(c_D) \textrm{.}
\]
The projection $\mfaut(c_D) \to \mbaEs(c_D)$ defined by this decomposition is given by
\[
    \xi^a \mapsto \phi_{ab} \xi^{b, a} - \tfrac{1}{2} \xi^a \phi_{ab,}^{\phantom{ab,} b} \textrm{.}
\]
\end{theorem}

Equation \eqref{equation:aEs-to-cKf} corrects a sign error in the statement of the theorem in \cite{HammerlSagerschnig2009}.

\subsection{Holonomy groups of Nurowski conformal structures}

Let $D$ be a generic $2$-plane field on an oriented $5$-manifold. Since the tractor metric $g^{\mcT}$ of the Nurowski conformal structure $c_D$ is indefinite but is equal to $H(\Phi)$, the discussion before \eqref{equation:induced-bilinear-form} implies that $\Phi$ is split-generic and that $g^{\mcT}$ has signature $(3, 4)$; in particular, this recovers the fact that $c$ has signature $(2, 3)$. Subsection \ref{subsection:G2} shows that the stabilizer in $\GL(\bbV)$ of a split-generic $3$-form is $\G_2$, so the characterization of holonomy containment in terms of tensor stabilizers in Subsection \ref{subsection:holonomy} lets us reformulate Theorem \ref{theorem:HS-Theorem-A} in the language of holonomy.

\begin{theorem}\label{theorem:Nurowski-holonomy-characterization}
A conformal structure $c$ on a connected, oriented $5$-manifold $M$ is Nurowski if and only if for some (equivalently any) $x \in M$, $\Hol_x(\nabla^{\mcT}) \leq \G_2$ for some copy of $\G_2$ contained in $\SO(g^{\mcT}_x)$.
\end{theorem}

Here, the containment is just the translation of the compatibility condition between the split-generic parallel tractor $3$-form and the tractor metric.

\begin{remark}
We can extend the statement of Theorem \ref{theorem:HS-Theorem-A} to nonoriented conformal structures on $5$-manifolds by replacing $M$ with its orientation cover. We can correspondingly extend the statement of Theorem \ref{theorem:Nurowski-holonomy-characterization} to such conformal structures by replacing $\G_2$ and $\SO(g^{\mcT}_x)$ with $\G_2 \cdot \bbZ_2$ and $\Ooperator(g^{\mcT}_x)$, respectively, where the nontrivial element of $\bbZ_2$ is just $-\id$ on the underlying vector space $\mcT_x$.
\end{remark}

If $D$ is real-analytic, so is the Nurowski conformal structure $c_D$ it induces, and thus we may associate to $D$ the canonical real-analytic metric $\wtg_D$ of $c_D$.

\begin{example}
For the $2$-plane fields $D_{F[\mba, b]}$ defined by the functions
\[
    F[\mba, b](x, y, p, q, z) := q^2 + a_6 p^6 + a_5 p^5 + a_4 p^4 + a_3 p^3 + a_2 p^2 + a_1 p + a_0 + b z \textrm{,}
\]
where $\mba := [a_0, \ldots, a_6]$ and $b$ are constants, Leistner and Nurowski computed the induced conformal structures and remarkably produced explicit corresponding Ricci-flat ambient metrics $\wtg_{F[\mba, b]}$ and compatible split-generic $3$-forms, which in particular shows that $\Hol(\wtg_{F[\mba, b]}) \leq \G_2$ for all $[\mba, b]$. Computing directly using the procedure given in \cite[\S5]{GrahamWillse2012G2} shows that if at least one of $a_3$, $a_4$, $a_5$, or $a_6$ is nonzero, then $D_{F[\mba, b]}$ has root type $[3, 1]$ everywhere except on at most three hyperplanes of the form $\set{p = p_0}$, where it has root type $[4]$ or $[\infty]$. Leistner and Nurowski showed for any such $[\mba, b]$ that $\Hol(\wtg_{F[\mba, b]}) = \G_2$ by eliminating the possibility of proper containment using ad hoc methods and explicit tensorial data they computed for this class.

(For completeness, if $a_3 = a_4 = a_5 = a_6 = 0$ but $a_2 \neq -\frac{2}{9} b^2, 2 b^2$, then $D_{F[\mba, b]}$ has constant root type $[4]$ and $\dim \mfaut(D_{F[\mba, b]}) = 7$, whence it is locally equivalent to $D_I$ for some constant $I$; see Subsection \ref{subsection:Cartan-class} below. If $a_3 = a_4 = a_5 = a_6 = 0$ and either $a_2 = -\frac{2}{9} b^2$ or $a_2 = 2 b^2$, then $D_{F[\mba, b]}$ has constant root type $[\infty]$, that is, $D_{F[\mba, b]}$ is locally flat.)
\end{example}

Graham and Willse used Theorem \ref{theorem:HS-Theorem-A} to show for all oriented, real-analytic $D$ that $\Hol(\wtg_D) \leq \G_2$, and extended Leistner and Nurowski's arguments to show that, in a suitable sense, equality holds for nearly all $D$ \cite{GrahamWillse2012G2}.

In the following we are interested in the exceptions thereto: In Section \ref{section:special-classes}, we will give a broad class of $2$-plane fields $D$ for which the conformal holonomy and the holonomy of suitable ambient metrics of the conformal structure $c_D$ are proper subgroups of $\G_2$.

To identify these groups explicitly, we will use the following proposition, which is just a translation of Proposition \ref{proposition:stabilizer-two-null-vectors} into the setting of holonomy using the stabilizer characterization of holonomy in Section \ref{subsection:holonomy}.

\begin{proposition}\label{proposition:holonomy-containment}
Let $(N, h)$ be a connected pseudo-Riemannian $7$-manifold that admits a parallel split-generic $3$-form $\Psi$ that satisfies $H(\Psi) = h$ and two linearly independent, parallel, null vector fields $\xi$ and $\eta$. Then, the holonomy of $h$ satisfies
\[
    \left\{
    \begin{array}{ll}
        \Hol(\nabla^h) \leq \Hoperator_5, & \textrm{if                           $\Psi(\xi, \eta, \, \cdot \,) =    0$;} \\
        \Hol(\nabla^h) \leq       \bbR^3, & \textrm{if $h(\xi, \eta)    = 0$ and $\Psi(\xi, \eta, \, \cdot \,) \neq 0$;} \\
        \Hol(\nabla^h) \leq \SL(2, \bbR), & \textrm{if $h(\xi, \eta) \neq 0$.}
    \end{array}
    \right.
\]
\end{proposition}

(See \cite[Corollary 2.5]{Kath1998} for a related result about common stabilizers of nonisotropic spinors and related results in 2.4 there.)

We can give an analogous result for tractor holonomy $\Hol(\nabla^{\mcT})$ by replacing $(N, h)$ with the tractor bundle $\mcT$ of a conformal structure and the tractor metric $g^{\mcT}$, $\Psi$ with a section of $\Lambda^3 \mcT^*$, and $\xi$ and $\eta$ with tractors. Though we do not need this result later, we record here a tractor holonomy version of part of this proposition in terms of base data as an application of the BGG splitting operators.

\begin{proposition}
Let $D$ be a generic $2$-plane field on a $5$-manifold $M$, and let $\phi \in \Gamma(\Lambda^2 T^* M [3])$ be the corresponding normal conformal Killing $2$-form. If $c$ admits two almost Ricci-flat scales $\sigma, \tau \in \Gamma(\mcD[1])$ such that neither is a constant multiple of the other (or equivalently, nonzero scales that correspond to non-homothetic metrics $\sigma^{-2} \mbg$ and $\tau^{-2} \mbg$) and such that
\[
    \phi_{ab} \sigma^a \tau^b - \tfrac{1}{4} \phi_{ab,}^{\phantom{ab,} a} (\sigma \tau^a - \tau \sigma^a) = 0 \textrm{,}
\]
then $\Hol(\nabla^{\mcT}) \leq \Hoperator_5$.
\end{proposition}


\section{Some special $2$-plane fields on $5$-manifolds}\label{section:special-classes}

\subsection{Cartan's class of highly symmetric $2$-plane fields with root type $[4]$}\label{subsection:Cartan-class}

In \cite{Cartan1910}, Cartan solved the local equivalence problem for generic $2$-plane fields on $5$-manifolds. He restricted attention to $2$-plane fields with constant root type and first classified $2$-plane fields according to those types. To avoid a proliferation of cases, he (largely) further restricted attention to $2$-plane fields whose symmetry algebra has dimension at least $6$. We summarize some of his results; for a relatively accessible but detailed exposition of Cartan's arguments, see \cite[\S17]{Stormark2000}.

\begin{theorem}\cite{Cartan1910} \label{theorem:Cartan-root-type-4}
Let $D$ be a generic $2$-plane field on a $5$-manifold $M$ with constant root type, and suppose that $\dim \mfaut(D) \geq 6$. Then, one of the following holds:
\begin{itemize}
    \item $D$ has constant root type $[\infty]$ and so is locally flat, and thus $\mfaut(D) \cong \mfg_2$
    \item $D$ has constant root type $[4]$ and $\dim \mfaut(D)$ is either $6$ or $7$
    \item $D$ has constant root type $[2, 2]$ and $\dim \mfaut(D) = 6$.
\end{itemize}
\end{theorem}

Cartan produced what he claimed was a local normal form for $D$ with constant root type $[4]$ and for which $\dim \mfaut(D) \geq 6$ \cite[\S9]{Cartan1910}: He claimed that for any point in $M$ one can find a neighborhood $U$ of that point and construct a principal bundle $E \to U$ and an adapted coframe $(\eta^1, \eta^2, \eta^3, \eta^4, \eta^5, \pi^1, \pi^2)$ on $E$ with structure equations
\begin{align}\label{equation:structure-equations}
    d\eta^1 &= 2 \eta^1 \wedge  \pi^1 + \eta^2 \wedge \pi^2 + \eta^3 \wedge \eta^4 \notag \\
    d\eta^2 &=   \eta^2 \wedge  \pi^1                       + \eta^3 \wedge \eta^5 \notag \\
    d\eta^3 &= I \eta^2 \wedge \eta^5 + \eta^3 \wedge \pi^1 + \eta^4 \wedge \eta^5 \notag \\
    d\eta^4 &= \tfrac{4}{3} I \eta^3 \wedge \eta^5 + \eta^4 \wedge \pi^1 + \eta^5 \wedge \pi^2 \\
    d\eta^5 &= 0 \notag \\
    d   \pi^1 &= 0 \notag \\
    d   \pi^2 &= - \pi^1 \wedge \pi^2 - I \eta^4 \wedge \eta^5 + \eta^2 \wedge \eta^5 \textrm{,} \notag
\end{align}
where $I$ is a smooth function on $U$; by construction, $D$ is the common kernel of the pullbacks of $\eta^1$, $\eta^2$, and $\eta^3$ to $M$ by any local section of $E \to U$. Pulling back the forms $\pi^1$ and $\pi^2$ shows that the Lie algebra of the structure group of $E$ is the unique nonabelian $2$-dimensional Lie algebra. Furthermore, Cartan constructed the frame so that the Cartan curvature of $D$ is just $A = (\eta^5)^4$. Differentiating both sides of the equations for $d\eta^3$ and $d\pi^2$ shows that $dI = J \eta^5$ for some function $J$. The invariant $I$ is fundamental in the sense that all other invariants, in particular $J$, are functions of $I$. In fact, Doubrov and Govorov have recently showed that Cartan's classification neglects at least one example: There is a generic $2$-plane field on a $5$-manifold of constant root type $[4]$ and with $6$-dimensional symmetry algebra but which cannot be realized locally in the above form for any function $I$; see Example \ref{example:doubrov-govorov} for further discussion of this counterexample.

Since $d\eta^5 = 0$, locally we may take $\eta^5 = dx$ for a coordinate $x$; substituting gives $dI = J \,dx$, so $I$ is a function of $x$ alone. One can then satisfy the pullback of the system \eqref{equation:structure-equations} to $M$ by an arbitrary section by setting
\begin{align*}
    \eta^1 &=  dz + \tfrac{7}{3} p I \,dy + q \,dp - [\tfrac{1}{2} q^2 + \tfrac{2}{3} I p^2 - \tfrac{1}{2} (1 + I^2 - I'')] dx \notag \\
    \eta^2 &=  dy - p \,dx \notag \\
    \eta^3 &= -dp + q \,dx \notag \\
    \eta^4 &=  dq - I \,dx \\
    \eta^5 &=  dx \notag \\
     \pi^1 &=  0  \notag \\
     \pi^2 &= -I' dy - \tfrac{4}{3} I \,dp + [(1 + I^2 - I'') y - \tfrac{4}{3} I' p - I q] dx \textrm{,}
\end{align*}
where we have suppressed pullback notation and suggestively used the variables that occur in the Monge normal form equation \eqref{equation:ODE}.

The general solution to the system defining the $2$-plane field $D$ (which again is the common kernel of the pullbacks of $\eta^1$, $\eta^2$, and $\eta^3$ to $M$) is
\begin{align*}
    y &= f  (x) \\
    p &= f' (x) \\
    q &= f''(x) \\
    z &= -\tfrac{1}{2} \int [ f''(x)^2 + \tfrac{10}{3} I(x) f'(x)^2 + (1 + I(x)^2 - I''(x)) f(x)^2 ] dx \textrm{.}
\end{align*}
(The coefficient $\frac{10}{3}$ corrects an arithmetic error in equation (6) in \cite[\S9]{Cartan1910}.) So, Cartan's claim would imply that any $2$-plane field with constant root type $[4]$ and symmetry algebra with dimension at least $6$ can be realized in Monge normal form \eqref{equation:local-coframe-partial} by the defining function
\[
    F_I(x, y, p, q, z) = -\tfrac{1}{2} [ q^2 + \tfrac{10}{3} I p^2 + (1 + I^2 - I'') y^2 ]
\]
for some smooth function $I(x)$. Conversely, given any smooth function $I$ on an open subset of $\bbR$, we denote by $D_I := D_{F_I}$ the generic $2$-plane field $\ker\set{\omega^1, \omega^2, \omega^3}$ that $F_I$ defines via \eqref{equation:local-coframe-partial}; in particular, $\partial_q^2 F_I = -1$, which is nowhere zero, so $F_I$ is generic.

If $I$ is constant, then \eqref{equation:structure-equations} defines the structure constants for the (local) symmetry algebra $\mfaut(D_I)$ of $D_I$, which thus has dimension $7$. (Kruglikov analyzes most of these in detail in \cite{Kruglikov} but instead realizes these $2$-plane fields using the Monge normal form equation defined by $F(x, y, p, q, z) = q^m$ for constants $m$. Together these $2$-plane fields [excepting those for $m = 0, 1$, which are nongeneric] account for all of the $2$-plane fields $D_I$ with constant $I$ except the single case $I = \pm \frac{3}{4}$, which is equivalent, for example, to the $2$-plane field defined by the Monge normal form equation $F(x, y, p, q, z) = \log q$ \cite{DoubrovGovorov2013a}.)

If $I$ is not constant, then since it is a function of $x$ we can regard $x$ itself as an invariant of the structure. This invariant turns out to be inessential (see \cite[\S14.5]{Stormark2000}), and so we can identify the structure equations of the symmetry algebra of $D_I$ by pulling back the structure equations \eqref{equation:structure-equations} to a leaf $\set{x = k}$; suppressing pullback notation, we get
\begin{align*}
    d\eta^1 &= 2 \eta^1 \wedge \pi^1 + \eta^2 \wedge \pi^2 + \eta^3 \wedge \eta^4 \\
    d\eta^2 &=   \eta^2 \wedge \pi^1                                              \\
    d\eta^3 &=   \eta^3 \wedge \pi^1 + \eta^4 \wedge \eta^5                       \\
    d\eta^4 &=   \eta^4 \wedge \pi^1                                              \\
    d \pi^1 &= 0                                                                  \\
    d \pi^2 &= \pi^2 \wedge \pi^1                                      \textrm{,}
\end{align*}
so in this case $\dim \mfaut(D_I) = 6$. Directly checking verifies Cartan's claim \cite[Subsection 44]{Cartan1910} that for constant $I$ the symmetry algebra is solvable, and that the same is true for nonconstant $I$.

\begin{remark}
For all smooth functions $I$, the symmetry algebra $\mfaut(D_I)$ contains the $2$-dimensional subalgebra $\ab{\partial_z, y \partial_y + p \partial_p + q \partial_q + 2 z \partial_z}$, but for general $I$ it is difficult to identify the other symmetries explicitly. With computer assistance one can identify for constant $I$ an explicit basis for the full $7$-dimensional symmetry algebra, but any such basis is surprisingly complicated, so we do not give one here.
\end{remark}

Nurowski's formula for a representative $g_F$ of the conformal structure $c_F$ induced by the $2$-plane field $D_F$ specializes dramatically for the functions $F_I$: Evaluating it (and rescaling by a constant for convenience) gives the representative metric
\begin{equation}\label{equation:I(x)-Nurowski-representative}
    g_I := g_{F_I} = -3 I (\omega^1)^2 + 3 \, \omega^1 \omega^4 - 10 I p \, \omega^1 \omega^5 - 3 \, \omega^2 \omega^5 - 2 (\omega^3)^2 \textrm{;}
\end{equation}
the coframe $(\omega^i)$ is that defined in \eqref{equation:local-coframe}. Here and henceforth, we identify a function $I(x)$ with its pullback to the space
\[
    M_I := \set{(x, y, p, q, z) : x \in \dom I}
\]
by the projection $(x, y, p, q, z) \mapsto x$.

We will show that the conformal structures $c_I := c_{F_I} = [g_I]$ determined by the functions $F_I$ all have tractor and ambient holonomy isomorphic to a particular proper subgroup of $\G_2$.

For use in our below proof of the main theorem, we give explicit formulae for some of the objects constructed above for the $2$-plane fields $D_I$.

\begin{proposition}\label{proposition:I(x)-parallel-3-form-data}
Let $I(x)$ be a smooth function whose domain is open and connected, and set $\wtM_I := \bbR_+ \times M_I \times \bbR$.

The generic $2$-plane field $D_I$ on the $5$-manifold $M_I$ is given by
\[
    D_I = \ab{\partial_q, \partial_x + p \partial_y + q \partial_p - \tfrac{1}{2} [ q^2 + \tfrac{10}{3} I p^2 + (1 + I^2 - I'') y^2 ] \partial_z} \textrm{,}
\]
and the Nurowski conformal structure $c_I$ it defines contains the representative $g_I$ defined by \eqref{equation:I(x)-Nurowski-representative}.

The metric $\wtg_I$ on $\wtM_I$ defined by
\begin{equation}\label{equation:I(x)-ambient-metric}
    \wtg_I
        = 2 \rho \,dt^2 + 2 t \,dt \,d\rho + t^2 \left( g_I - \tfrac{2}{3} \rho I \,dx^2 \right)
\end{equation}
is a Ricci-flat ambient metric for $c_I$ (we suppress the notation for pulling back $g_I$ by the projection $\Pi: \wtM_I \to M_I$ defined by $\Pi: (t, u, \rho) \mapsto u$).

The $3$-form $\wtPhi_I \in \Gamma(\Lambda^3 T^* \wtM_I)$ defined by
\begin{multline*}
    \wtPhi_I := C [ -   9 t^2     \,dt     \wedge \omega^1 \wedge \omega^2
                    -   2 t^2     \,dt     \wedge \omega^3 \wedge d\rho
                    -   3 t^3     \omega^1 \wedge \omega^3 \wedge \omega^4 \\
                    \qquad +  10 t^3 I p \omega^1 \wedge \omega^3 \wedge \omega^5
                    -     t^3 I   \omega^1 \wedge \omega^5 \wedge d\rho
                    +   3 t^3     \omega^2 \wedge \omega^3 \wedge \omega^5 \\
                    +     t^3     \omega^4 \wedge \omega^5 \wedge d\rho
                    +
                    (
                    -   3 t^2 I \,dt \wedge \omega^1 \wedge \omega^5
                    +     t^2     dt \wedge \omega^4 \wedge \omega^5
                    ) \rho
                    ]
\end{multline*}
is parallel and satisfies $H(\wtPhi_I) = \wtg_I$, where $C = 2^{-5 / 6} 3^{-1 / 3}$, and where we suppress the notation for the pullback of $I$ and the coframe forms $\omega^a$ by $\Pi$. In particular, $\wtPhi_I$ is split-generic.

The parallel tractor $3$-form associated to $D_I$ is $\Phi_I := \wtPhi_I \vert_{\mcG}$, and the trivialization of the normal conformal Killing form $\phi \in \Gamma(\Lambda^2 T^* M[3])$ with respect to $g_I$ is
\begin{equation} \label{equation:I(x)-defining-2-form}
    \phi_I := -9 C \, \omega^1 \wedge \omega^2 \textrm{.}
\end{equation}
\end{proposition}
\begin{proof}
The formula for $D_I$ is given in Subsection \ref{subsection:ordinary-differential-equations}. The metric $\wtg_I$ has homogeneity $2$ by construction, and inspection shows that pull it back to $\mcG$ yields $\mbg_0 = t^2 g_I$, so it is a pre-ambient metric for $c_I$; computing directly shows that it is Ricci-flat, so it is in fact an ambient metric for $c_I$. Computing directly shows that $\wtPhi_I$ is parallel and satisfies $H(\wtPhi_I) = \wtg_I$, and $\phi_I = \Pi_0(\Phi_I)$.
\end{proof}

\begin{proposition}\label{proposition:I(x)-parallel-null-vector-data}
Let $I(x)$ be a smooth function on an open interval. The functions $\sigma(x)$ in the $2$-dimensional solution space $\mcS$ of the homogeneous, linear, second order ordinary differential equation
\begin{equation}\label{equation:I(x)-Ricci-flat-ODE}
    \sigma'' - \tfrac{1}{3} I \sigma = 0 \textrm{,}
\end{equation}
are almost Ricci-flat scales of $c_I$ (trivialized with respect to the representative $g_I$). In particular, every tractor $L_0(\sigma)$ in the corresponding $2$-dimensional vector subspace $L_0(\mcS) \subset \Gamma(\mcT)$ is null.

For each solution $\sigma \in S$, the vector field
\begin{equation}\label{equation:I(x)-parallel-null-vector-field}
    \xi^{\sigma} := d(\sigma t)^{\sharp} = t^{-1}(-\tfrac{2}{3} \sigma' \partial_z + \sigma \partial_{\rho}) \in \Gamma(T\wtM_I)
\end{equation}
is parallel and null, the tractor $\xi^{\sigma} \vert_{\mcG}$ produced by restricting it is just $L_0(\sigma)$, and the infinitesimal symmetry of $c_I$ corresponding to $\sigma$ via \eqref{equation:aEs-to-cKf} is given in the frame \eqref{equation:local-frame} by
\[
    -\tfrac{1}{9}(\sigma E_3 + 4 \sigma' E_4) \textrm{.}
\]
In the definition of $\xi^{\sigma}$, we have suppressed the notation for the pullback of $\sigma$ by the projection $\bbR_+ \times M_I \times \bbR \to M_I$.
\end{proposition}
\begin{proof}
The trivialized Ricci-flat scales are exactly the functions $\sigma$ of $(x, y, p, q, z)$ such that $\Ric (\sigma^{-2} g_I \vert_{M - \Sigma})$, where $\Sigma$ is the zero set of $\sigma$. Consider functions $\sigma$ that depend only on $x$; computing directly gives that on $M - \Sigma$,
\[
    \Ric(\sigma^{-2} g_I) = 3 \sigma^{-1} (\sigma'' - \tfrac{1}{3} I \sigma) dx^2 \textrm{,}
\]
where a prime $'$ denotes the derivative $\partial_x$. For such $\sigma$, the Ricci curvature vanishes if and only if the quantity $\sigma'' - \tfrac{1}{3} I \sigma$ does. The remaining claims follow from direct computation.
\end{proof}

\begin{remark}
In fact, direct (but tedious) analysis of the equation $\Ric(\sigma^{-2} g_I \vert_{M - \Sigma}) = 0$ for general $\sigma$ (that is, not just those that depend only on $x$) shows that these account for all of the almost Einstein scales of $g_I$, but we will see this follows indirectly from below results.
\end{remark}

With the above data in hand, we are prepared to prove the main results.

\begin{proof}[Proof of Theorem A]
By Proposition \ref{proposition:I(x)-parallel-3-form-data} $\wtg_I$ admits a parallel split-generic $3$-form $\wtPhi_I$ that satisfies $H(\wtPhi_I) = \wtg_I$, and by Proposition \ref{proposition:I(x)-parallel-null-vector-data} it admits two linearly independent parallel null vector fields, say, $\xi$ and $\eta$. Computing gives that $\wtPhi(\xi, \eta, \cdot) = 0$, so by Proposition \ref{proposition:holonomy-containment}, $\Hol(\wtg_I) \leq \Hoperator_5$.

We now show that the infinitesimal holonomy $\mfhol'(\wtg_I)$ of the ambient metric $\wtg_I$ has dimension at least $5$. Then, since $\Hol'(\wtg_I) \leq \Hol(\wtg_I)$, both holonomy groups are equal to $\Hoperator_5$.

We compute the infinitesimal holonomy of $\wtg_I$ using the definition; its curvature is
\begin{equation}\label{equation:I(x)-ambient-curvature}
    \wtR = 60 t^2 (\omega^1 \wedge \omega^5)^2 \textrm{.}
\end{equation}
Fix $u \in \wtM_I$; then, $\mfhol'_u(\wtg_I)$ admits a filtration $(V_u^r)$ by the vector spaces $V_u^r \subset \End(T_u \wtM)$ spanned by endomorphisms generated by at most $r$ derivatives of curvature: More precisely, $V_u^r = V^r \vert_u$, where
\begin{multline*}
    V^r := \set{\wtR_{AB \phantom{C} D, J_1 \cdots J_k}^{\phantom{AB} C} X^A Y^B Z_1^{J_1} \cdots Z_k^{J_k} : k \leq r; X, Y, Z_1, \ldots, Z_k \in \Gamma(T\wtM_I)} \\
                \subseteq \Gamma(\End(T\wtM_I)) \textrm{.}
\end{multline*}

We compute the filtered pieces one at a time. Respectively the definition and an easy induction using the Leibniz rule give
\[
    \left\{
        \begin{array}{ll}
            V^0 = \set{\wtR_{AB \phantom{C} D}^{\phantom{AB} C} X^A Y^B : X, Y \in \Gamma(T\wtM_I)} \\
            V^r = V^{r - 1} \cup \set{S^C_{\phantom{C} D, J} Z^J : S \in V^{r - 1}; Z \in \Gamma(T\wtM_I)}, r > 0
        \end{array}
    \right. \textrm{.}
\]

Now, consulting \eqref{equation:I(x)-ambient-curvature} and raising an index shows that $V^0 = \ab{ \psi_1 }$, where
\[
    \psi_1 := E_2 \ot \omega^1 + E_4 \ot \omega^5 \in \Gamma(\End(T\wtM_I)) \textrm{,}
\]
and $\ab{ \, \cdot \, }$ denotes the span over $C^{\infty}(T\wtM_I)$.

Next, computing $(\psi_1)^C_{\phantom{c} D, J}$ and contracting with an arbitrary vector field $Z^J$ gives that $V^1 = \ab{\psi_1, \psi_2, \psi_3}$, where
\begin{align*}
    \psi_2 &:= t^{-1} E_2 \ot dt + 15 \partial_{\rho} \ot \omega^5 \\
    \psi_3 &:= 4 E_2 \ot \omega^3 - 3 E_3 \ot \omega^5 - 6 t^{-1} E_4 \ot dt + 90 \partial_{\rho} \ot \omega^1 \textrm{.}
\end{align*}
Continuing gives that $V^2 = \ab{\psi_1, \psi_2, \psi_3, \psi_4}$, where
\begin{multline*}
    \psi_4 := 3 E_1 \ot \omega^5 + 3 E_2 \ot \omega^4 - 10 I p E_2 \ot \omega^5 + 9 t^{-1} E_3 \ot dt + 6 I E_4 \ot \omega^5 + 180 \partial_{\rho} \ot \omega^3 \textrm{.}
\end{multline*}
and that $V^3 = \ab{\psi_1, \psi_2, \psi_3, \psi_4, \psi_5}$, where
\begin{multline*}
    \psi_5 := t^{-1} E_1 \ot dt + t^{-1} I E_4 \ot dt + 15 I \partial_{\rho} \ot \omega^1 - 15 \partial_{\rho} \ot \omega^4 + 50 I p \partial_{\rho} \ot \omega^5 \textrm{.}
\end{multline*}
In particular, the given generating set of $V^3$ is linearly independent at every point in $\wtM$, so $\dim \mfhol'_u(\wtg_I) \geq \dim V^3_u = 5$. (In fact, since the Lie algebra of $\Hol_u(\wtg_I)$ has dimension at most $5$, this gives $\mfhol_u(\wtg_I) = \mfhol'_u(\wtg_I) = V^3_p$.)

If we take $u$ to be a point in $\mcG$, then Theorem \ref{theorem:CGGH} shows that $\mfhol'(\nabla^{\mcT}_I) = \mfhol'(\wtg_I) = \mfh_5$. Since $M_I$ is simply connected, $\Hol(\nabla^{\mcT}_I) = \Hoperator_5$.

We could have avoided using Theorem \ref{theorem:CGGH} and instead showed directly the equality of the infinitesimal holonomy algebras by expanding the derivatives $\wtR_{ab \phantom{C}D, J_1 \cdots J_k}^{\phantom{ab} C}$ to third order in Christoffel symbols; using the relationship between the tractor and ambient curvature tensors, one can then show that the restrictions of $\psi_a$ to $\mcG$, $1 \leq a \leq 5$, which by construction are sections of $\End(\mcT)$, can all be produced by taking derivatives of tractor curvature.
\end{proof}

\begin{corollary}\label{corollary:almost-Einstein-scale-space}
Let $I$ be a smooth function on an open interval. Then, the space $\mbaEs(c_I)$ of almost Einstein scales of $D_I$ is exactly the $2$-dimensional space $\mcS$ of almost Ricci-flat scales identified in Proposition \ref{proposition:I(x)-parallel-null-vector-data}. Then, the dimension of the conformal symmetry algebra of $c_I$ is
\[
\dim \mfaut(c_I) =
    \left\{
        \begin{array}{ll}
            8, & \textrm{$I$ not constant} \\
            9, & \textrm{$I$ constant}
        \end{array}
        \right.
        \textrm{.}
\]
\end{corollary}
\begin{proof}
Analyzing the representation $\bbV$ of $\G_2$ in Subsection \ref{subsection:G2} shows that the vectors preserved by the restriction of that representation to $\Hoperator_5 = \Stab_{\G_2} (e_1) \cap \Stab_{\G_2} (e_2)$ are exactly those in the $2$-dimensional subspace $\ab{e_1, e_2}$. Since $\Hol(\nabla^{\mcT}_I) = \Hoperator_5$, the space of parallel sections of $\mcT$ is $2$-dimensional, and by the correspondence \eqref{equation:bijection-parallel-tractors-almost-Einstein-scales}, $\mbaEs(c_I)$ has dimension $2$ and hence must coincide with $\mcS$.

Now, Theorem \ref{theorem:HS-Theorem-B} gives that
\[
    \dim \mfaut(c_I) = \dim \mfaut(D_I) + \dim \mbaEs(c_I) = \dim \mfaut(D_I) + 2 \textrm{.}
\]
By the discussion at the beginning of the section, the symmetry algebra $\mfaut(D_I)$ has dimension $7$ if $I$ is constant and dimension $6$ if not.
\end{proof}

\begin{remark}
Metrics admitting the types of parallel objects the metrics $\wtg_I$ do admit many additional parallel objects. Let $(N, h)$ be a connected pseudo-Riemannian $7$-manifold that admits a parallel split-generic $3$-form $\Psi$ that satisfies $H(\Psi) = h$ and a special null plane field $S$ comprised of parallel vector fields, then $S$ itself is parallel, as is the conull $5$-plane field $\smash{S^{\perp} \supset S}$. Given any nonzero parallel null vector field $\xi \in \Gamma(S)$, all plane fields in the complete flag field
\[
    0 \subset [\xi] \subset S \subset \Ann \xi \subset (\Ann \xi)^{\perp} \subset S^{\perp} \subset [\xi]^{\perp} \subset TN
\]
on $N$ are parallel (and null or conull) and hence are subrepresentations of $\Hol(h)$; here, $\Ann \xi$ is the $3$-plane field with fiber $(\Ann \xi)_x := \Ann (\xi_x)$. Given a second parallel null vector field that is not a multiple of $\xi$, one can produce further parallel plane fields by forming the intersections and spans (and the orthogonal plane fields thereof) of the pieces of the corresponding flag fields.

If we fix $u \in N$ and identify $\xi_u$ with $e_1 \in \bbV$ and $S$ with $\ab{e_1, e_2}$, then
\begin{align*}
    \Ann \xi         _u &= \ab{e_1, e_2, e_3               } \textrm{,}     \\
    \Ann \xi ^{\perp}_u &= \ab{e_1, e_2, e_3, e_4          } \textrm{,}     \\
           S ^{\perp}_u &= \ab{e_1, e_2, e_3, e_4,      e_6} \textrm{, and} \\
        [\xi]^{\perp}_u &= \ab{e_1, e_2, e_3, e_4, e_5, e_6} \textrm{.}
\end{align*}
Then, consulting \eqref{equation:h5-representation} shows that even though the holonomies of the metrics $\wtg_I$ and the corresponding connections $\nabla^{\mcT}_I$ do not act irreducibly, they do act indecomposably.
\end{remark}

\begin{remark}
Let $I$ be a smooth function on an open interval, and let $\xi$ be a parallel vector field on $(\wtM_I, \wtg_I)$ (see \eqref{equation:I(x)-parallel-null-vector-field}). Again if we fix $u \in \wtM_I$ and identify $\xi_u$ with $e_1 \in \bbV$, then $\smash{[\xi_u]^{\perp} = \ab{e_1, e_2, e_3, e_4, e_5, e_6}}$, and consulting \eqref{equation:h5-representation} again shows that this subrepresentation of $\Hoperator_5$ is faithful but that no proper subrepresentation thereof is. (Restricting to $\mcG \subset \wtM_I$ yields the analogous statement for parallel tractors and tractor holonomy.) Since $[\xi_u]$ is null, however, the pullback of $\wtg_I$ to any leaf $L$ of the foliation defined by the plane field $\smash{[\xi]^{\perp}}$ (which is integrable because it is parallel) via the inclusion $\iota: L \hookrightarrow \wtM_I$ is degenerate: By construction, at each point $u \in L$, $\xi_u$ is in $T_u L$ and is orthogonal to every vector in that space.

The $2$-form $\iota^* \wtg_I$ degenerates only along this direction, however, so it descends to a pseudo-Riemannian metric $g$ on the space of integral curves of $\xi\vert_L \in \Gamma(TL)$. The representation $\smash{\Hol_u(\wtg_I) \vert_{[\xi_u]^{\perp}}}$ fixes $\xi_u$, so it descends to a representation on the quotient space $\smash{[\xi_u]^{\perp} / [\xi_u]}$, which by construction we may identify with the holonomy representation $\Hol_{[u]}(g)$ of $g$ at the integral curve $[u]$ through $u$. If we yet again identify $\xi_u$ with $e_1 \in \bbV$, then $[\xi_u]^{\perp} / [\xi_u] \cong \ab{e_2, e_3, e_4, e_5, e_6}$,
and consulting \eqref{equation:h5-representation} shows that $\Hol_{[u]}(g) \cong \bbR^3$.

Now, let $S$ be the parallel $2$-plane field comprising the null parallel vector fields of $\wtg_I$. By construction, each plane field in the complete flag field
\[
    0 \subset S / [\xi]  \subset \Ann \xi / [\xi] \subset (\Ann \xi)^{\perp} / [\xi] \subset S^{\perp} / [\xi] \subset TL / [\xi]
\]
on the space of integral curves is parallel (and null or conull). (In fact, \eqref{equation:h5-representation} shows that every $2$-, $3$-, or $4$-plane $P_u$ such that $\smash{S_u / [\xi_u] \subseteq P_u \subseteq S^{\perp}_u / [\xi_u]}$ extends to a parallel $k$-plane field.)

Again, $\smash{S^{\perp}_u / [\xi_u]}$ is a faithful representation of $\Hoperator_5$ but no proper subrepresentation is. Furthermore the pullback of $g$ to any leaf of the foliation determined by the plane field $\smash{S^{\perp} / [\xi]}$ is again degenerate, but at each point in the leaf, the pullback degenerates only along the direction $S / [\xi]$. So, it descends to a metric on the ($3$-dimensional) space of integral curves of the line field $S / [\xi]$, which we may also interpret as the space of leaves of the foliation determined by the plane field $S \vert_R$ in a leaf $R$ of the foliation determined by $S^{\perp} \subset TM$, and again consulting \eqref{equation:h5-representation} shows that the holonomy of this metric is trivial.
\end{remark}

One can interpret the ideas in the previous remark to determine the metric holonomy of distinguished representatives of the conformal classes $c_I$.

\begin{proposition}
Let $I$ be a smooth function on an open interval, and let $g \in c_I$ be a Ricci-flat representative. Then, $g$ admits a parallel null vector field and $\Hol(g) \cong \bbR^3$.
\end{proposition}
\begin{proof}
Let $\sigma \in \Gamma(\mcD[1])$ be the Ricci-flat scale so that $g = \sigma^{-2} \mbg$ (in particular, $\sigma$ vanishes nowhere, and so by changing sign if necessary, we may assume that $\sigma$ is everywhere positive). By Proposition \ref{proposition:I(x)-parallel-null-vector-data}, $\sigma$ must be in the $2$-dimensional vector space identified therein, and the vector field $\xi^{\sigma} \in \Gamma(T\wtM_I)$ is parallel with respect to the Levi-Civita connection of $\wtg_I$. Then, the orthogonal plane field $\smash{[\xi^{\sigma}]^{\perp}}$ is $\ker (\xi^{\sigma})^{\flat} = \ker d(t \sigma)$ (where, as in that proposition, $\sigma$ has been trivialized by the representative $g_I \in c_I$), and the leaves of the foliation defined by this plane field are the hypersurfaces $L_C := \set{t = C / \sigma(x)}$, $C > 0$. On  $L_C$, the integral curve $[u_0](\tau)$ of $\xi^{\sigma} \vert_{L_C}$ satisfying the arbitrary initial condition $[u_0](0) = u_0 := (C / \sigma(x_0), x_0, y_0, p_0, q_0, z_0, \rho_0)$ is
\[
    \gamma(\tau) = (C / \sigma(x_0), x_0, y_0, p_0, q_0, z_0 - \tfrac{1}{3 C} \sigma(x_0) \sigma'(x_0) \tau, \rho_0 + \tfrac{1}{C} \sigma(x_0)^2 \tau) \textrm{,}
\]
which is defined for all time $\tau$. In particular, every integral curve $[u_0]$ intersects the hypersurface $L_C \cap \mcG = L \cap \set{\rho = 0}$ exactly once, so we may identify it with the space of integral curves $[u]$, but by construction $L_C \cap \mcG$ is the image of the metric $g$ regarded as a section $M \to \mcG$. Unwinding definitions shows that (1) the induced metric on the space of integral curves of $\xi^{\sigma} \vert_L$ is just the pullback of $\wtg_I$ to $L_C \cap \mcG$, and (2) if we identify $M$ with $L_C \cap \mcG$, this pullback is just $g$ itself.
\end{proof}

The referee observed that parts of the proof of this proposition can be simplified some using some easy facts about the ambient metrics of Ricci-flat metrics, including that for such a metric $g$, there are coordinates $r$ and $s$ on the ambient space for which the metric $\wtg = 2 dr\, ds + r^2 g$ is an ambient metric for $[g]$, where as usual we suppress pullback notation.

\subsection{Plane fields defined by ODEs $z' = F(y'')$}\label{subsection:F(q)-plane-fields}

Many of the above results hold just as well for the class of $2$-plane fields $D_{F(q)}$ defined via \eqref{equation:local-coframe-partial} for the smooth functions $F(q)$ that depend only on $q$ and for which $F''(q)$ is nowhere zero, so that $D_{F(q)}$ is generic; by Subsection \ref{subsection:ordinary-differential-equations}, these $2$-plane fields encode ordinary differential equations $z' = F(y'')$, where $y$ and $z$ are functions of $x$. (Nurowski considered this class of $2$-plane fields as an example in \cite{Nurowski2005}, essentially gave equations \eqref{equation:F(q)-representative metric} and \eqref{equation:F(q)-Ricci-flat-ODE}, and observed that for generic functions $F(q)$ the root type of $D_{F(q)}$ is generically equal to $[4]$; see below.) We identify $F(q)$ with its pullback to the set
\[
    M_{F(q)} := \set{(x, y, p, q, z) : q \in \dom F}
\]
by the projection $(x, y, p, q, z) \mapsto q$. These $2$-plane fields again all have symmetry algebra with dimension at least $6$, and we can now identify these symmetries explicitly:
\begin{multline}\label{equation:F(q)-symmetry-algebra}
    \mfaut(D_{F(q)}) \geq \bigg\langle \partial_x, \partial_y, \partial_z, x \partial_x + 2 y \partial_y + p \partial_p + z \partial_z,
                                x \partial_y + \partial_p, \\
                                F' \partial_x + (p F' - z) \partial_y  + (q F' - F) \partial_p + \int F'' F \, dq \cdot \partial_z \bigg\rangle  \textrm{.}
\end{multline}
As mentioned earlier, if $D$ is a generic $2$-plane field for which $\dim \mfaut(D) = 7$, then $D$ is locally equivalent either to $D_{q^m}$ for some constant $m$ or to $D_{\log q}$. Conversely, if $m \not\in \set{-1, 0, \frac{1}{3}, \frac{2}{3}, 1, 2}$, then $\dim \mfaut(D_{q^m}) = 7$ and the symmetry algebra is spanned by the right-hand side of \eqref{equation:F(q)-symmetry-algebra} and $y \partial_y + p \partial_p + q \partial_q + m z \partial_z$ \cite{Kruglikov}; the symmetry algebra of $D_{\log q}$ is spanned by the right-hand side of \eqref{equation:F(q)-symmetry-algebra} and $y \partial_y + p \partial_p + q \partial_q + x \partial_z$. If $m \in \set{-1, \frac{1}{3}, \frac{2}{3}, 2}$, then $D_{q^m}$ is locally flat, and if $m = 0$ or $m = 1$, then $\partial_q^2 (q^m) = 0$ and so $D_{q^m}$ is not generic.

Computing directly using the procedure given in \cite[\S5]{GrahamWillse2012G2}, the Cartan curvature of the $2$-plane field $D_{F(q)}$ is
\begin{equation}\label{equation:F(q)-Cartan-curvature}
    A_{F(q)} = (F'')^{-4} \Psi[F''] \, dq^4 \textrm{,}
\end{equation}
where $\Psi: C^{\infty}(\dom F) \to C^{\infty}(\dom F)$ is the nonlinear differential operator
\begin{equation}\label{equation:Psi-operator}
    \Psi[U] := 10 U^{(4)} U^3 - 80 U''' U' U^2 - 51 (U'')^2 U^2 + 336 U'' (U')^2 U - 224 (U')^4 \textrm{.}
\end{equation}
So, these $2$-plane fields are closely related to the class of $2$-plane fields $D_I$ described above, but the root type of a $2$-plane field $D_{F(q)}$ need not be constant: $D_{F(q)}$ has root type $[\infty]$ at $(x, y, p, q, z) \in M_{F(q)}$ if $\Psi[F''](q) = 0$ and root type $[4]$ at that point otherwise. To the knowledge of the author, it is unknown whether every $D_I$ can be locally realized (at each point) as a $2$-plane field $D_{F(q)}$ for some $F(q)$.

For completeness, we collect some explicit data for the $2$-plane fields $D_{F(q)}$ in two propositions; they are produced in the same way as are their analogues in Propositions \ref{proposition:I(x)-parallel-3-form-data} and \ref{proposition:I(x)-parallel-null-vector-data}, so we suppress the proofs.

\begin{proposition}\label{proposition:F(q)-parallel-3-form-data}
Let $F(q)$ be a function on an open interval such that $F''$ vanishes nowhere, and set $\wtM_{F(q)} := \bbR_+ \times M_{F(q)} \times \bbR$.

The generic $2$-plane field $D_{F(q)}$ on the $5$-manifold $M_{F(q)}$ is given by
\[
    D_{F(q)} = \ab{\partial_q, \partial_x + p \partial_y + q \partial_p + F(q) \partial_z} \textrm{,}
\]
and the Nurowski conformal structure $c_{F(q)}$ it defines contains the (again, relatively simple) representative
\begin{multline}\label{equation:F(q)-representative metric}
    g_{F(q)} = 30 (F'')^4 \omega_1 \omega_4 + [- 3 F^{(4)} F'' + 4 (F''')^2] \omega_2^2 \\
                - 10 F''' (F'')^2 \omega_2 \omega_3 + 30 (F'')^3 \omega_2 \omega_5 - 20 (F'')^4 \omega_3^3 \textrm{.}
\end{multline}

The metric $\wtg_{F(q)}$ on $\wtM_{F(q)}$ defined by
\begin{equation}\label{equation:F(q)-ambient-metric}
    \wtg_{F(q)}
        = 2 \rho \,dt^2 + 2 t \,dt \,d\rho + t^2 \left( g_{F(q)} - \rho \frac{17 F^{(4)} F'' - 56 (F''')^2}{5 (F'')^2} \, (\omega^4)^2 \right)
\end{equation}
is a Ricci-flat ambient metric for $c_{F(q)}$ (we suppress the notation for pulling back $g_{F(q)}$ and $\omega^4$ by the projection $\Pi: \wtM_{F(q)} \to M$ defined by $\Pi: (t, x, \rho) \mapsto x$).

The $3$-form $\wtPhi_{F(q)} \in \Gamma(\Lambda^3 T^* \wtM_{F(q)})$ defined by
\begin{multline*}
    \wtPhi_{F(q)} := C'[                     (F'')^5      t^2 dt \wedge \omega^1 \wedge \omega^2
                        +     \tfrac{1}{  9} F'''         t^2 dt \wedge \omega^2 \wedge d\rho \\
                            - \tfrac{1}{ 45} (F'')^2      t^2 dt \wedge \omega^3 \wedge d\rho
                        +     \tfrac{5}{  3} F''' (F'')^4 t^3 \omega^1 \wedge \omega^2 \wedge \omega^4 \\
                            - \tfrac{1}{  3} (F'')^6      t^3 \omega^1 \wedge \omega^3 \wedge \omega^4
                            - \tfrac{1}{  3} (F'')^5      t^3 \omega^2 \wedge \omega^3 \wedge \omega^5 \\
                        +     \tfrac{1}{900} (F^{(4)} - 168 (F''')^2 (F'')^{-1})
                                                          t^3 \omega^2 \wedge \omega^4 \wedge d\rho \\
                            + \tfrac{7}{ 90} F''' F''     t^3 \omega^3 \wedge \omega^4 \wedge d\rho
                            + \tfrac{1}{ 90} (F'')^2      t^3 \omega^4 \wedge \omega^5 \wedge d\rho \\
                        + [
                              \tfrac{1}{900} (103 F^{(4)} - 504 (F''')^2 (F'')^{-1}) t^2 dt \wedge \omega^2 \wedge \omega^4 \\
                            + \tfrac{7}{ 90} F''' F''                                t^2 dt \wedge \omega^3 \wedge \omega^4
                            + \tfrac{1}{ 90} (F'')^2                                 t^2 dt \wedge \omega^4 \wedge \omega^5
                          ] \rho
                       ] \textrm{,}
\end{multline*}
is parallel and satisfies $H(\wtPhi_{F(q)}) = \wtg_{F(q)}$, where $C' = 2^{1 / 3} 3^{5 / 3} 5^{3 / 2}$, and where we suppress the notation for the pullback of $F$ and the coframe forms $\omega^a$ by $\Pi$. In particular, $\wtPhi_{F(q)}$ is split-generic.

The parallel tractor $3$-form associated to $D_{F(q)}$ is $\Phi_{F(q)} := \wtPhi_{F(q)} \vert_{\mcG}$, and the trivialization of the normal conformal Killing form $\phi \in \Gamma(\Lambda^2 T^* M[3])$ with respect to $g_{F(q)}$ is
\begin{equation} \label{equation:F(q)-defining-2-form}
    \phi_{F(q)} := C' (F'')^5 \omega^1 \wedge \omega^2 \textrm{.}
\end{equation}
\end{proposition}

The form of the ambient metric \eqref{equation:F(q)-ambient-metric} is different in \cite{Nurowski2008}; there, Nurowski starts with a different representative metric and uses supplemental variables $t$ and $u := -\rho t$ on the ambient space instead of $t$ and $\rho$.

\begin{proposition}\label{proposition:F(q)-parallel-null-vector-data}
Let $F(q)$ be a function on an open interval such that $F''$ vanishes nowhere. The functions $\sigma$ in the $2$-dimensional solution space $S$ of the homogeneous, linear, second order ordinary differential equation
\begin{equation}\label{equation:F(q)-Ricci-flat-ODE}
    10 (F'')^2 \sigma'' - 40 F''' F'' \sigma' + (-17 F^{(4)} F'' + 56 (F''')^2) \sigma = 0 \textrm{,}
\end{equation}
are almost Ricci-flat scales of $c_{F(q)}$ (that have been trivialized with respect to the representative $g_{F(q)}$). In particular, every tractor $L_0(\sigma)$ in the corresponding ($2$-dimensional) vector subspace $L_0(S) \subset \Gamma(\mcT)$ is null.

For each solution $\sigma$, the vector field
\[
    \xi^{\sigma} := d(\sigma t)^{\sharp} = \tfrac{1}{15} (F'')^{-4} t^{-1} \sigma' \partial_y + t^{-1} \sigma \partial_{\rho} \in \Gamma(T\wtM_{F(q)})
\]
is parallel, and the tractor $\xi^{\sigma} \vert_{\mcG}$ produced by restricting it is just $L_0(\sigma)$, where we have suppressed the notation for the pullback of $\sigma$ by the projection $\bbR_+ \times M_{F(q)} \times \bbR \to M_{F(q)}$.
\end{proposition}

(Equation \eqref{equation:F(q)-Ricci-flat-ODE} could also be recovered by the corresponding Ricci-flatness equation in \cite[\S3]{Nurowski2008} by rescaling the variable $\Upsilon$ there by the appropriate conformal factor and then changing variables via $\Upsilon = -\log \sigma$.)

We also give an analog of Theorem A for the $2$-plane fields determined by functions $F(q)$.

\begin{theorem}\label{theorem:F(q)-holonomy}
Let $F(q)$ be a function on an open interval such that $F''$ vanishes nowhere, and let $\wtg_{F(q)}$ be the Ricci-flat ambient metric \eqref{equation:F(q)-ambient-metric} of $c_{F(q)}$.
Let $\nabla^{\mcT}_{F(q)}$ denote the tractor connection of $c_{F(q)}$.
\begin{itemize}
    \item If $A_{F(q)} =    0$, that is, if $D_{F(q)}$ is locally flat, then $\Hol(\nabla^{\mcT}_{F(q)}) \cong \Hol(\wtg_{F(q)}) \cong \set{e}$.
    \item If $A_{F(q)} \neq 0$, then $\Hol(\nabla^{\mcT}_{F(q)}) \cong \Hol(\wtg_{F(q)}) \cong \Hoperator_5$.
\end{itemize}
\end{theorem}
\begin{proof}
By Proposition \ref{proposition:F(q)-parallel-3-form-data} $\wtg_{F(q)}$ admits a parallel split-generic $3$-form $\wtPhi_{F(q)}$ that satisfies $H(\wtPhi_{F(q)}) = \wtg_{F(q)}$, and by Proposition \ref{proposition:F(q)-parallel-null-vector-data} it admits two linearly independent parallel null vector fields, say, $\xi$ and $\eta$. Computing gives that $\wtPhi(\xi, \eta, \cdot) = 0$, so by Proposition \ref{proposition:holonomy-containment}, $\Hol(\wtg_{F(q)}) \leq \Hoperator_5$.

Computing gives that the curvature of $\wtM_{F(q)}$ is
\begin{equation*}
    \wtR = \tfrac{3}{5} t^2 (F'')^{-2} \Psi[F''] (\omega_2 \wedge \omega_4)^2 \textrm{,}
\end{equation*}
where $\Psi$ is the differential operator given by \eqref{equation:Psi-operator}. By \eqref{equation:F(q)-Cartan-curvature}, if $A_{F(q)} = 0$ then $\Psi[F''] = 0$ and so $\wtR = 0$. Since $\wtM_{F(q)}$ is simply connected, $\Hol(\wtg_{F(q)}) = \set{e}$, and since $\Hol(\nabla^{\mcT}_{F(q)}) \leq \Hol(\wtg_{F(q)})$, $\Hol(\nabla^{\mcT}_{F(q)}) = \set{e}$.

If $\Psi[F''] \neq 0$, pick $u \in M_{F(q)}$ such that $\Psi[F''](q) \neq 0$. Proceeding as in the proof of Theorem A, one can show that $\dim V^3_u = 5$, so $\dim \mfhol'_u(\wtg_{F(q)}) \geq 5$, and thus $\Hol(\wtg_{F(q)}) = \Hoperator_5$. Again, Theorem \ref{theorem:CGGH} gives that $\Hol(\nabla^{\mcT}_{F(q)}) = \Hoperator_5$ too.
\end{proof}

\subsection{Other examples}

Some further examples of $2$-plane fields reveal constraints on the possible relationships among root type $[4]$, symmetry algebra dimension, and holonomy.

The first example shows that having symmetry algebra of dimension at least $6$ is not a necessary condition for the holonomy of the tractor connection to be equal to $\Hoperator_5$.

\begin{example}
Example 6 of \cite{Nurowski2005} states that, according to Cartan, every $2$-plane field with root type $[4]$ can be (presumably locally) realized as $D_{F(q)}$ for some function $F(q)$. The is untrue: Consider the $2$-plane fields $D_{F[r]}$ defined for constant $r$ by
\[\label{equation:Strazzullo-examples}
    F[r](x, y, p, q, z) = e^y \left[ 1 + (e^{-2 y} q - \tfrac{1}{2} e^{-2 y} p^2)^r \right]
\]
on $\set{2 q > p^2}$; except when $r = 0, 1$ these $2$-plane fields are generic, and computing directly shows that
\begin{multline}\label{equation:Strazzullo-examples-symmetries}
    \mfaut(D_{F[r]}) \geq \langle \partial_x, x \partial_x - \partial_y - p \partial_p - 2 q \partial_q, \\
                        x^2 \partial_x - 2 x \partial_y - 2 (x p + 1) \partial_p - 2 (p + 2 x q) \partial_q, \partial_z\rangle \cong \mfgl(2, \bbR) \textrm{,}
\end{multline}
Strazzullo computed that if $r \in \set{-1, 2}$ \cite[Example 6.7.1]{Strazzullo2009} then $D_{F[r]}$ has root type $[4]$, and tedious analysis shows that for these values equality holds in \eqref{equation:Strazzullo-examples-symmetries} and hence in particular that $\dim \mfaut(D_{F[r]}) = 4$. Since $\dim \mfaut(D_{F(q)}) \geq 6$ for all functions $F(q)$, the $2$-plane fields $D_{F[-1]}$ and $D_{F[2]}$ are not locally equivalent to $D_{F(q)}$ for any $F$, nor to $D_I$ for any function $I$. One can still find for both of these examples, however, an explicit Ricci-flat ambient metric for the induced conformal class and, proceeding as in the proofs of Theorems A and \ref{theorem:F(q)-holonomy}, show that the tractor connection and the ambient metric both have holonomy $\Hoperator_5$. In particular, this example suggests that there might be a much broader class of $2$-plane fields with associated holonomy groups equal to $\Hoperator_5$ than the classes $D_I$ and $D_{F(q)}$ considered in this paper.
\end{example}

\begin{example}\label{example:doubrov-govorov}[Doubrov \& Govorov's Counterexample]
Recently Doubrov and Govorov constructed the $2$-plane field $D^* := D_F$ defined by the function \cite{DoubrovGovorov2013a}
\[
    F(x, y, p, q, z) = y + q^{1 / 3} \textrm{.}
\]
They computed that $D^*$ has constant root type $[4]$ and that $\mfaut(D^*) \cong \mfsl(2, \bbR) \rtimes \mfh_3$, where $\mfh_3$ denotes the $3$-dimensional Heisenberg algebra; in particular, $\mfaut(D^*)$ is nonsolvable and $\dim \mfaut(D^*) = 6$. All of the $2$-plane fields $D_I$ have solvable symmetry algebra, however, so $D^*$ is not locally equivalent to $D_I$ for any function $I$, which disproves Cartan's longstanding claim that the $2$-plane fields $D_I$ (locally) exhaust the $2$-plane fields of constant root type $[4]$ and symmetry algebra of dimension at least $6$.

Strazzullo claims that $D_{F[2 / 3]}$ also has root type $[4]$ and $6$-dimensional symmetry algebra \cite[Example 6.7.2]{Strazzullo2009}; this, together with the fact $\mfaut(D_{F[2 / 3]})$ contains a subalgebra isomorphic to $\mfsl(2, \bbR)$ and so is nonsolvable, would mean that it, too, would be a counterexample to Cartan's claim, though apparently this was not noticed until later. Computing, however, gives that $D_{F[2 / 3]}$ has constant root type $[2, 1, 1]$ and that the claimed symmetry algebra is not correct (in fact, this would violate Theorem \ref{theorem:Cartan-root-type-4}). Presumably this is a typo: Checking shows that the $2$-plane field $D_{F[1 / 3]}$ does have constant root type $[4]$ and the ($6$-dimensional) symmetry algebra indicated for $D_{F[2 / 3]}$---the algebra is spanned by the right-hand of \eqref{equation:Strazzullo-examples-symmetries} and
\begin{multline*}
    \langle
        e^{- y / 2} [4 \partial_x + 2 p^2 \partial_p + (6 p q - p^3) \partial_q - 4 e^y \partial_z], \\
        e^{- y / 2} [4 x \partial_x - 8 \partial_y + 2 x p^2 \partial_p + (6 x p q - 4 q - p^3 x + 2p^2) \partial_q - 4 e^y x \partial_z]
    \rangle
\end{multline*}
---so $D_{F[1 / 3]}$ is also a counterexample to Cartan's claim. In fact, Doubrov and Govorov state \cite[Remark 3]{DoubrovGovorov2013a} that they will prove in a forthcoming paper \cite{DoubrovGovorov2013b} that up to local equivalence, $D^*$ is the unique $2$-plane field with transitive symmetry algebra of dimension at least $6$ that Cartan did not identify. Checking shows that the symmetry algebra of $D_{F[1 / 3]}$ is transitive, and so by that result $D_{F[1 / 3]}$ is locally equivalent to $D^*$.

This is the only example of which the author is aware of a $2$-plane field of constant root type $[4]$ for which the associated holonomy groups are $\G_2$; in particular, it shows that constant root type $[4]$ is not a sufficient condition for having holonomy equal to $\Hoperator_5$, even if one also assumes that the symmetry algebra has dimension at least $6$. It remains possible, however, that having root type $[4]$ or $[\infty]$ at each point is a necessary condition for an (oriented) $2$-plane field to have associated holonomy groups equal to $\Hoperator_5$, or even for those holonomy groups to be a proper subgroup of $\G_2$.
\end{example}

\begin{remark}
Analyzing the Nurowski conformal structure $c^*$ induced by $D^*$ reveals that its behavior differs substantively from that of the structures $c_I$ induced by the $2$-plane fields $D_I$ in Cartan's class and also enjoys additional unusual properties. First, the induced conformal structure $c^* := c_{D^*}$ is not almost Einstein, which, by Theorem \ref{theorem:HS-Theorem-B} implies that $\mfaut(c^*) = \mfaut(D^*)$, and by the discussion in \ref{subsection:holonomy} that the tractor and ambient holonomy groups associated to $c^*$ are not contained in the stabilizer in $\G_2$ of any nonzero vector in the standard representation. Remarkably, with substantial effort one can solve explicitly for the Ricci-flat ambient metric $\wtg^*$ of $c^*$---there are relatively few known classes of examples of conformal structures that are not almost Einstein for which this is true. (Moreover the exact expression one most easily obtains for $\wtg^*$ is not polynomial in $\rho$.) Using this expression one can show using the techniques in this paper that the tractor and ambient holonomy groups associated to $c^*$ are in fact the full group $\G_2$, and hence this yields another explicit example of a metric with this exceptional holonomy group. Because these features are of independent interest, we postpone further discussion of (and explicit data for) this unusual example to a dedicated article currently in preparation \cite{Willse2013}.
\end{remark}

\bibliographystyle{alpha}
\bibliography{symmetric_235_h5_holonomy}

\end{document}